\documentclass[a4paper,twoside]{article}
\usepackage{a4}
\usepackage{amssymb}
\usepackage{amsmath}
\usepackage{upref}
\usepackage[active]{srcltx}
\allowdisplaybreaks[2] % To make it possible for displaybreaks
\usepackage[dvips,colorlinks,citecolor=blue,linkcolor=blue]{hyperref}
\usepackage[dvipsnames]{color}
\newcount\minutes \newcount\hours
\hours=\time
\divide\hours 60
\minutes=\hours
\multiply\minutes -60
\advance\minutes \time
\newcommand{\klockan}{\the\hours:{\ifnum\minutes<10 0\fi}\the\minutes}
\newcommand{\tid}{\today\ \klockan}
\newcommand{\prtid}{\smash{\raise 10mm \hbox{\LaTeX ed \tid}}}
\renewcommand{\prtid}{}
\makeatletter
\def\sectionmark#1{} %\markboth{{\sectnr #1}}{{\sectnr #1}}} %Journal
\def\subsectionmark#1{}
\newcommand{\sectnr}{\ifnum \c@secnumdepth >\z@
                 \thesection.\hskip 1em\relax \fi}
\def\@evenhead{\footnotesize\rm\thepage\hfil\leftmark\hfil\llap{\prtid}}
\def\@oddhead{\footnotesize\rm\rlap{\prtid}\hfil\rightmark\hfil\thepage}
\def\tableofcontents{\section*{Contents} %\@mkboth{Contents}{Contents}} %Journal
 \@starttoc{toc}}
\makeatother
\makeatletter
\def\@biblabel#1{#1.}
\makeatother
\makeatletter
\let\Thebibliography=\thebibliography
\renewcommand{\thebibliography}[1]{\def\@mkboth##1##2{}\Thebibliography{#1}
\addcontentsline{toc}{section}{References}
\frenchspacing % Maybe not needed
\setlength{\@topsep}{0pt}% Delete if extra space before list
\setlength{\itemsep}{0pt}%
\setlength{\parskip}{0pt plus 2pt}%
}
\makeatother
\makeatletter
\def\mdots@{\mathinner.\nonscript\!.%
 \ifx\next,.\else\ifx\next;.\else\ifx\next..\else
 \nonscript\!\mathinner.\fi\fi\fi}
\let\ldots\mdots@
\let\cdots\mdots@
\let\dotso\mdots@
\let\dotsb\mdots@
\let\dotsm\mdots@
\let\dotsc\mdots@
\def\vdots{\vbox{\baselineskip2.8\p@ \lineskiplimit\z@
    \kern6\p@\hbox{.}\hbox{.}\hbox{.}\kern3\p@}}
\def\ddots{\mathinner{\mkern1mu\raise8.6\p@\vbox{\kern7\p@\hbox{.}}%
    \raise5.8\p@\hbox{.}\raise3\p@\hbox{.}\mkern1mu}}
\makeatother
\makeatletter
\let\Enumerate=\enumerate
\renewcommand{\enumerate}{\Enumerate%
\setlength{\@topsep}{0pt}% Delete if extra space before list
\setlength{\itemsep}{0pt}%
\setlength{\parskip}{0pt plus 1pt}%
\renewcommand{\theenumi}{\textup{(\alph{enumi})}}%
\renewcommand{\labelenumi}{\theenumi}%
}
\let\endEnumerate=\endenumerate
\renewcommand{\endenumerate}{\endEnumerate\unskip}
\makeatother
\makeatletter
\def\@seccntformat#1{\csname the#1\endcsname.\quad}
\makeatother
\makeatletter
\long\def\@makecaption#1#2{%
  \vskip\abovecaptionskip
  \sbox\@tempboxa{ #1. #2}%
  \ifdim \wd\@tempboxa >\hsize
    #1. #2\par
  \else
    \global \@minipagefalse
    \hb@xt@\hsize{\hfil\box\@tempboxa\hfil}%
  \fi
  \vskip\belowcaptionskip}
\makeatother
\newcommand{\authortitle}[2]{\author{#1}\title{#2}\markboth{#1}{#2}}
\newcommand{\art}[6]{{\sc #1, \rm #2, \it #3 \bf #4 \rm (#5), \mbox{#6}.}}
\newcommand{\artin}[3]{{\sc #1, \rm #2,  in #3.}}
\newcommand{\artnopt}[6]{{\sc #1, \rm #2, \it #3 \bf #4 \rm (#5), \mbox{#6}}}
\newcommand{\artprep}[3]{{\sc #1, \rm #2, \it #3.}}
\newcommand{\arttoappear}[3]{{\sc #1, \rm #2, to appear in \it #3}}
\newcommand{\auth}[2]{{#1, #2.}}
\newcommand{\book}[3]{{\sc #1, \it #2, \rm #3.}}
\newcommand{\AND}{{\rm and }}
\RequirePackage{amsthm}
\newtheoremstyle{descriptive}%
  {\topsep}   %{\medskipamount}          % Space above
  {\topsep}   %  {\medskipamount}          % Space below
  {\rmfamily} % Body font
  {}          % Indent
  {\bfseries} % Head font
  {.}         % Punctuation after thm head
  { }         % Space after thm head
  {}          % Thm head spec(?)
\newtheoremstyle{propositional}%
  {\topsep}   %  {\medskipamount}          % Space above
  {\topsep}   %  {\medskipamount}          % Space below
  {\itshape}  % Body font
  {}          % Indent
  {\bfseries} % Head font
  {.}         % Punctuation after thm head
  { }         % Space after thm head
  {}          % Thm head spec(?)
\theoremstyle{propositional}
\newtheorem{thm}{Theorem}[section]
\newtheorem{prop}[thm]{Proposition}
\newtheorem{lem}[thm]{Lemma}
\newtheorem{cor}[thm]{Corollary}
\theoremstyle{descriptive}
\newtheorem{deff}[thm]{Definition}
\newtheorem{example}[thm]{Example}
\newtheorem{remark}[thm]{Remark}
\newtheorem{openprob}[thm]{Open problem}
\makeatletter
\renewenvironment{proof}[1][\proofname]{\par
  \pushQED{\qed}%
  \normalfont 
  \trivlist
  \item[\hskip\labelsep
        \itshape
    #1\@addpunct{.}]\ignorespaces
}{%
  \popQED\endtrivlist\@endpefalse
}
\makeatother
\newcommand{\setm}{\setminus}
\renewcommand{\subsetneq}{\varsubsetneq}
\renewcommand{\emptyset}{\varnothing}
\def\vint{\mathop{\mathchoice%
          {\setbox0\hbox{$\displaystyle\intop$}\kern 0.22\wd0%
           \vcenter{\hrule width 0.6\wd0}\kern -0.82\wd0}%
          {\setbox0\hbox{$\textstyle\intop$}\kern 0.2\wd0%
           \vcenter{\hrule width 0.6\wd0}\kern -0.8\wd0}%
          {\setbox0\hbox{$\scriptstyle\intop$}\kern 0.2\wd0%
           \vcenter{\hrule width 0.6\wd0}\kern -0.8\wd0}%
          {\setbox0\hbox{$\scriptscriptstyle\intop$}\kern 0.2\wd0%
           \vcenter{\hrule width 0.6\wd0}\kern -0.8\wd0}}%
          \mathopen{}\int}
\newcommand{\Cp}{{C_p}}
\newcommand{\bCp}{{\protect\itoverline{C}_p}}
\DeclareMathOperator{\diam}{diam}
\newcommand{\grad}{\nabla}
\DeclareMathOperator{\dist}{dist}
\DeclareMathOperator{\inner}{inner}
\DeclareMathOperator{\Lip}{Lip}
\DeclareMathOperator*{\essliminf}{ess\,lim\,inf}
\DeclareMathOperator*{\essinf}{ess\,inf}
\newcommand{\bdry}{\partial}
\newcommand{\bdy}{\bdry}
\newcommand{\loc}{_{\rm loc}}
{\catcode`p =12 \catcode`t =12 \gdef\eeaa#1pt{#1}}      % Get slantfactor
\def\accentadjtext#1{\setbox0\hbox{$#1$}\kern   % Convert it with height
                \expandafter\eeaa\the\fontdimen1\textfont1 \ht0 }
\def\accentadjscript#1{\setbox0\hbox{$#1$}\kern % Convert it with height
                \expandafter\eeaa\the\fontdimen1\scriptfont1 \ht0 }
\def\accentadjscriptscript#1{\setbox0\hbox{$#1$}\kern   % Convert it with height
                \expandafter\eeaa\the\fontdimen1\scriptscriptfont1 \ht0 }
\def\accentadjtextback#1{\setbox0\hbox{$#1$}\kern       % Convert it with height
                -\expandafter\eeaa\the\fontdimen1\textfont1 \ht0 }
\def\accentadjscriptback#1{\setbox0\hbox{$#1$}\kern     % Convert it with height
                -\expandafter\eeaa\the\fontdimen1\scriptfont1 \ht0 }
\def\accentadjscriptscriptback#1{\setbox0\hbox{$#1$}\kern % Convert it with height
                -\expandafter\eeaa\the\fontdimen1\scriptscriptfont1 \ht0 }
\def\itoverline#1{{\mathsurround0pt\mathchoice
        {\rlap{$\accentadjtext{\displaystyle #1}
                \accentadjtext{\vrule height1.593pt}
                \overline{\phantom{\displaystyle #1}
                \accentadjtextback{\displaystyle #1}}$}{#1}}
        {\rlap{$\accentadjtext{\textstyle #1}
                \accentadjtext{\vrule height1.593pt}
                \overline{\phantom{\textstyle #1}
                \accentadjtextback{\textstyle #1}}$}{#1}}
        {\rlap{$\accentadjscript{\scriptstyle #1}
                \accentadjscript{\vrule height1.593pt}
                \overline{\phantom{\scriptstyle #1}
                \accentadjscriptback{\scriptstyle #1}}$}{#1}}
        {\rlap{$\accentadjscriptscript{\scriptscriptstyle #1}
                \accentadjscriptscript{\vrule height1.593pt}
                \overline{\phantom{\scriptscriptstyle #1}
                \accentadjscriptscriptback{\scriptscriptstyle #1}}$}{#1}}}}
\def\itunderline#1{{\mathsurround0pt\mathchoice
        {\rlap{$\underline{\phantom{\displaystyle #1}
                \accentadjtextback{\displaystyle #1}}$}{#1}}
        {\rlap{$\underline{\phantom{\textstyle #1}
                \accentadjtextback{\textstyle #1}}$}{#1}}
        {\rlap{$\underline{\phantom{\scriptstyle #1}
                \accentadjscriptback{\scriptstyle #1}}$}{#1}}
        {\rlap{$\underline{\phantom{\scriptscriptstyle #1}
                \accentadjscriptscriptback{\scriptscriptstyle #1}}$}{#1}}}}
\newcommand{\al}{\alpha}
\newcommand{\alp}{\alpha}
\newcommand{\be}{\beta}
\newcommand{\ga}{\gamma}
\newcommand{\dmu}{d\mu}
\newcommand{\de}{\delta}
\newcommand{\eps}{\varepsilon}
\newcommand{\La}{\Lambda}
\newcommand{\Om}{\Omega}
\newcommand{\clOmprim}{{\overline{\Om}\mspace{1mu}}'}
\newcommand{\clOmm}{{\overline{\Om}\mspace{1mu}}^M}
\newcommand{\clGM}{{\itoverline{G}\mspace{1mu}}^M}
\newcommand{\clOmGm}{{\overline{\Om}\mspace{1mu}}^{G}}  
\newcommand{\clOmP}{{\overline{\Om}\mspace{1mu}}^P}
\newcommand{\OmGm}{{\Om^{G}}}  
\newcommand{\clOm}{{\overline{\Om}}}
\newcommand{\Omm}{{\Om^M}}
\newcommand{\Ommo}{{\Om^M_0}}
\newcommand{\bdym}{{\bdy_M}}
\newcommand{\bdyGM}{{\bdy_{\clGM}}}
\renewcommand{\phi}{\varphi}
\newcommand{\p}{{$p\mspace{1mu}$}}   
\newcommand{\R}{\mathbf{R}}
\newcommand{\eR}{{\overline{\R}}}
\newcommand{\C}{\mathbf{C}}
\newcommand{\Kt}{\widetilde{K}}
\newcommand{\Ga}{\Gamma}
\newcommand{\Np}{N^{1,p}}
\newcommand{\Nploc}{N^{1,p}\loc}
\newcommand{\Lploc}{L^p\loc}
\newcommand{\Hp}{P}                 % upper=lower solution
\newcommand{\Hpind}[1]{P_{#1}}      % with domain
\newcommand{\uHpind}[1]{\itoverline{P}_{#1}}      % with domain
\newcommand{\lHpind}[1]{\itunderline{P}_{#1}}      % with domain
\newcommand{\oHp}{H}                % old H_p from BBS
\newcommand{\oHpind}[1]{H_{#1}}     % with domain
\newcommand{\A}{\mathcal{A}}%
\newcommand{\At}{\widetilde{\mathcal{A}}}%
\newcommand{\K}{\mathcal{K}}%
\newcommand{\UU}{\mathcal{U}}%
\newcommand{\LL}{\mathcal{L}}%
\newcommand{\gat}{{\widetilde{\gamma}}}
\newcommand{\ut}{\tilde{u}}
\newcommand{\gt}{\tilde{g}}
\newcommand{\Ut}{\widetilde{U}}
\newcommand{\ft}{\tilde{f}}
\newcommand{\din}{d_{\inner}}
\newcommand{\gahat}{\widehat{\ga}}
\newcommand{\dMto}{\overset{d_M}\longrightarrow}
\newcommand{\dGMto}{\overset{\dGM}\longrightarrow}
\newcommand{\dM}{d_M}
\newcommand{\dGM}{d_{G}} 

\numberwithin{equation}{section}
\newenvironment{ack}{\medskip{\it Acknowledgement.}}{}

\begin{document}

\authortitle{Anders Bj\"orn, Jana Bj\"orn
	and Nageswari Shanmugalingam}
{The Dirichlet problem for \p-harmonic functions with respect to
the Mazurkiewicz boundary}
\title{The Dirichlet problem for \p-harmonic functions with respect to
the Mazurkiewicz boundary, and new capacities}
\author{
Anders Bj\"orn \\
\it\small Department of Mathematics, Link\"opings universitet, \\
\it\small SE-581 83 Link\"oping, Sweden\/{\rm ;}
\it \small anders.bjorn@liu.se
\\
\\
Jana Bj\"orn \\
\it\small Department of Mathematics, Link\"opings universitet, \\
\it\small SE-581 83 Link\"oping, Sweden\/{\rm ;}
\it \small jana.bjorn@liu.se
\\
\\
Nageswari Shanmugalingam\footnote{Corresponding author}
\\
\it \small  Department of Mathematical Sciences, University of Cincinnati, \\
\it \small  P.O.\ Box 210025, Cincinnati, OH 45221-0025, U.S.A.\/{\rm ;}
\it \small shanmun@uc.edu
}

\date{}
\maketitle

\noindent{\small
{\bf Abstract} 
In this paper we develop the Perron method for solving the Dirichlet problem
for the analog of the \p-Laplacian,
i.e.\ for \p-harmonic functions, 
with  Mazurkiewicz boundary values.  
The setting considered here is that of metric spaces, 
where the boundary of the 
domain in question is replaced with the Mazurkiewicz boundary. 
Resolutivity for Sobolev  and continuous functions, 
as well as invariance results for perturbations on small sets,
are obtained. We use these results to improve the known resolutivity and invariance 
results for functions on the standard (metric) boundary.
We also illustrate the results of this paper by
discussing several examples.
} 

\bigskip
\noindent
{\small \emph{Key words and phrases}: 
capacity, Dirichlet problem, doubling measure, finite connectivity at the boundary, 
inner metric,  Mazurkiewicz distance, 
metric space, nonlinear potential theory, 
perturbation invariance, Poincar\'e inequality,
 \p-energy minimizer, \p-harmonic function,
Perron method, resolutive.
}


\section{Introduction}

When considering the Dirichlet problem 
on the slit disc $B((0,0),1) \setm [0,1] \subset \R^2$
it is quite natural to allow for two boundary conditions at 
each point in the slit $(0,1]$ (apart from the tip),
one from above and one from below.

Our aim in this paper is to 
give a general approach (via the Mazurkiewicz boundary)
suitable for solving this generalized \p-harmonic Dirichlet problem for a large
class of domains in $\R^n$, as well as in metric spaces. We 
develop the  Perron method for the Mazurkiewicz boundary and 
obtain  resolutivity and invariance results for it,
which also improve a number of older results for 
the given boundary. These results are new even in the
(unweighted) Euclidean setting.

A continuous function $u$ is \emph{\p-harmonic}, $1<p<\infty$,
if it locally minimizes the \p-energy integral
\[
      \int |\nabla u|^p \, dx.
\]
When $p=2$ this reduces to the classical 
harmonic functions, while
for $p \ne 2$ it is  the main prototype for a
nonlinear elliptic equation.
To generalize this to metric spaces, 
the concept of upper gradients (introduced
by Heinonen and Koskela in \cite{HeKo98}) is used, leading to a
variational definition similar to the one above. In such a general setting,
there is no corresponding equation.
The nonlinear potential theory associated with
\p-harmonic functions has been studied for half a century,
first on $\R^n$ and then in various other situations (manifolds,
Heisenberg groups, graphs etc.). 
The metric space theory is more recent and was first
considered by Shanmugalingam~\cite{Sh-rev}. It gives
a unified treatment covering most of the earlier cases.
For further development see the monographs
Heinonen--Kilpel\"ainen--Martio~\cite{HeKiMa} (for weighted $\R^n$)
and Bj\"orn--Bj\"orn~\cite{BBbook} (for metric spaces)
and the references therein.

To describe our approach, let $\Om$ be a bounded domain in $\R^n$.
We define the \emph{Mazurkiewicz distance} $\dM$ on $\Om$ by
\[
     \dM(x,y) =\inf_E \diam E,
\]
where the infimum is taken over all connected sets $E \subset \Om$
containing $x,y \in \Om$.
(The Mazurkiewicz distance was first used by
Mazurkiewicz~\cite{mazurkiewicz16} in 1916,
but goes under different names in the literature,
see Remark~\ref{rmk-M-dist}.)
The \emph{Mazurkiewicz boundary} is then defined using the completion
of $(\Om,\dM)$.
In the slit disc this leads to the desired boundary
with two boundary points corresponding to each point in the slit,
while for smooth domains it coincides with the usual boundary.
In this paper we study Perron solutions for \p-harmonic
functions with respect to the Mazurkiewicz boundary.

In the Perron method, given a function $f$ on the boundary, an 
upper and a lower Perron solution are constructed using superharmonic
(or \p-superharmonic) functions. When these two solutions
coincide they give a reasonable solution to the Dirichlet problem
with $f$ as boundary values (for harmonic or
\p-harmonic functions).
The function $f$ is then called \emph{resolutive},
and the solution is denoted by $\Hp f$.

The Perron method was introduced by Perron~\cite{perron} (and independently
by Remak~\cite{remak}) in 1923 for harmonic functions. 
Wiener and Brelot made important contributions in the linear case (on $\R^n$)
leading to Brelot's resolutivity result~\cite{brelot}, which says that
the resolutive functions are exactly the $L^1$ functions with respect to
harmonic measure. 
For this reason the method is often called the Perron--Wiener--Brelot method,
but since this is less appropriate in the nonlinear situation we prefer
to name it just after Perron.

For \p-harmonic functions on $\R^n$ the Perron method was first
studied by Granlund--Lindqvist--Martio~\cite{GLM86}.
Kilpel\"ainen~\cite{Kilp89} showed that continuous
functions are resolutive,
and 
Heinonen--Kilpel\"ainen--Martio~\cite{HeKiMa}
adapted the proof to  weighted $\R^n$.
Using a different approach this
result was further generalized to
metric spaces in Bj\"orn--Bj\"orn--Shanmugaglingam~\cite{BBS2}.
Therein it was also shown that restrictions of quasicontinuous
representatives of Sobolev functions on the entire metric space (e.g.\ $\R^n$)
are resolutive.
Moreover, if $f$ is either such a Sobolev function or 
$f \in C(\bdy \Om)$,
and $h=f$ q.e.\ (i.e.\ outside 
a set of capacity zero),
then $h$ is also resolutive and $\Hp h=\Hp f$.
In this paper we improve upon both these results.

To be able to generalize this theory to the Mazurkiewicz boundary,
we need the Mazurkiewicz boundary 
(or rather the Mazurkiewicz closure) to be compact 
(although, the paper Estep--Shanmugalingam~\cite{ES}
avoids the compactness requirement for the prime end boundary). 
It was shown by Karmazin~\cite{karmazin2008} that this happens if and only if
the domain is finitely connected at the boundary
(for a more elementary and self-contained proof of this fact 
see Bj\"orn--Bj\"orn--Shanmugaglingam~\cite{BBStop}).
Under this assumption we generalize all the results from \cite{BBS2}
mentioned above to the Mazurkiewicz boundary, with some additional improvements.
To do so we consider a new relative capacity, $\bCp$, adapted to the
topology that connects the domain to its Mazurkiewicz boundary.
With this new capacity, 
we consider how the boundary looks from inside $\Om$, not from 
the underlying metric space or even within the closure of $\Om$.

Using the new capacity we  also improve upon the results in~\cite{BBS2} when 
considering the usual Perron solutions with respect to the given metric.
In the invariance results there, we replace the usual Sobolev capacity with the smaller
capacity $\bCp$, thus allowing for perturbations on larger sets.
We also obtain resolutivity for more functions.

For harmonic functions on the slit disc there are two classical approaches: the prime
end boundary introduced by Carath\'eodory~\cite{car} in 1913 and the Martin
boundary introduced by Martin~\cite{martin} in 1941.

The minimal Martin kernel, a generalization of the Poisson kernel, is built 
for the harmonic equation,
and gives a very well suited boundary for the linear harmonic 
Dirichlet problem. First proposed by Martin~\cite{martin}, this notion was
developed further by many, including Ancona~\cite{Anc1}, \cite{Anc2} and
Anderson--Schoen~\cite{AndSch}.
There are \p-harmonic generalizations of the Martin boundary,
see e.g.\ Lewis--Nystr\"om~\cite{lewisNysAnnals},
but unlike in the linear case $p=2$, 
these generalizations are not connected to integral representations
of solutions to the Dirichlet problem.

The prime end boundary is instead constructed directly from the 
geometry of the domain and does not rely on any underlying equation.
Carath\'eodory's original approach works very well for simply 
(and finitely) connected planar domains.
Over the years there have been many suggestions for extending
prime ends to more general situations,
by different people and with different applications in mind,
see the discussion in \cite{ABBSprime}. Recently 
Adamowicz--Bj\"orn--Bj\"orn--Shan\-mu\-ga\-lin\-gam~\cite{ABBSprime}
gave a definition of prime end boundary suitable for a large class
of domains in metric spaces.
For domains which are finitely connected at the boundary
the Mazurkiewicz boundary is homeomorphic to the prime end boundary
of \cite{ABBSprime},  see \cite[Corollary~10.9]{ABBSprime}.
Thus our results with respect to the Mazurkiewicz boundary can equivalently
be formulated for the prime end boundary in such domains.
In the special case of the topologist's comb, further improvements
upon the general results of this paper are given in A.~Bj\"orn~\cite{ABcomb}.

The outline of the paper is as follows: 
After giving a survey of background results
from first-order analysis on metric spaces in Section~\ref{sect-prelim},
we introduce the new capacity $\bCp$ in Section~\ref{sect-newcap}
and the Mazurkiewicz distance in Section~\ref{sect-dM}.
Section~\ref{sect-NpOmm} is devoted to Sobolev spaces 
with respect to the Mazurkiewicz distance.
The necessary background theory on \p-harmonic and superharmonic
functions is given in Section~\ref{sect-superharm}, making
it possible to define Perron solutions with respect to the
Mazurkiewicz boundary in the subsequent section.

Our main resolutivity result is given in Theorem~\ref{thm-Newt-resolve-Omm}.
In Section~\ref{sect-bdyM-Om} we use Theorem~\ref{thm-Newt-resolve-Omm}
to show resolutivity  of continuous functions and to obtain
invariance results for perturbations along the lines indicated above.
New resolutivity and invariance results with respect to the given 
metric are described in Section~\ref{sect-resol-bdryOm}, and some 
further generalizations of the results
in Sections~\ref{sect-Perron-Omm} and~\ref{sect-bdyM-Om} are
given in Section~\ref{sect-gen-Perron}.
Section~\ref{sect-examples} is devoted to a number of examples
showing how our results can be applied.
We end the paper with an appendix comparing the
different capacities used in this paper.

\begin{ack}
This research started while the first two authors visited
the University of Cincinnati during the first half year of 2010,
and continued while the last author visited Link\"opings universitet
in March 2011. We wish to thank Tomasz Adamowicz for fruitful discussions.

The first two authors were supported by the Swedish Research Council.
The first author was also a Fulbright scholar during his
visit to the University of Cincinnati, supported by the Swedish
Fulbright Commission, while the second author was a Visiting Taft Fellow
during her visit to the University of Cincinnati,  supported by the Charles Phelps 
Taft Research Center at the University of Cincinnati.
The third author was also partially supported by the Taft Research Center, 
the Simons Foundation grant \#200474, and the NSF grant DMS-1200915.
\end{ack}

\section{Notation and preliminaries}
\label{sect-prelim}

We assume throughout the paper that $1 \le p<\infty$ 
and that $X=(X,d,\mu)$ is a metric space equipped
with a metric $d$ and a positive complete  Borel  measure $\mu$ 
such that $0<\mu(B)<\infty$ for all balls $B \subset X$
(we adopt the convention that balls are nonempty and open).
We emphasize that the $\sigma$-algebra on which $\mu$ is defined
is obtained by the completion of the Borel $\sigma$-algebra.
It follows that $X$ is separable.

We also assume that $\Om \subset X$ is a nonempty open 
set. Further standing assumptions will be given at the
end of Section~\ref{sect-newcap} and 
at the beginning of subsequent sections.

A \emph{curve} is a continuous mapping from an interval,
and a \emph{rectifiable} curve is a curve with finite length.
We will only consider curves which are nonconstant, compact and rectifiable.
A curve can thus be parameterized by its arc length $ds$. 

We follow Heinonen and Koskela~\cite{HeKo98} in introducing
upper gradients as follows (they called them very weak gradients).

\begin{deff} \label{deff-ug}
A nonnegative Borel function $g$ on $X$ is an \emph{upper gradient} 
of an extended real-valued function $f$
on $X$ if for all (nonconstant, compact and rectifiable) curves  
$\gamma: [0,l_{\gamma}] \to X$,
\begin{equation} \label{ug-cond}
        |f(\gamma(0)) - f(\gamma(l_{\gamma}))| \le \int_{\gamma} g\,ds,
\end{equation}
where we follow the convention that the left-hand side is $\infty$ 
whenever both terms therein are infinite.
If $g$ is a nonnegative measurable function on $X$
and if (\ref{ug-cond}) holds for \p-almost every curve (see below), 
then $g$ is a \emph{\p-weak upper gradient} of~$f$. 
\end{deff}

Here we say that a property holds for \emph{\p-almost every curve}
if it fails only for a curve family $\Ga$ with zero \p-modulus, 
i.e.\ there exists $0\le\rho\in L^p(X)$ such that 
$\int_\ga \rho\,ds=\infty$ for every curve $\ga\in\Ga$.
Note that a \p-weak upper gradient \emph{need not} be a Borel function,
only measurable. Given that the underlying measure $\mu$ is Borel regular,
every measurable function $g$ can be modified on a set of measure zero
to obtain a Borel function, from which it follows that 
$\int_{\gamma} g\,ds$ is defined (with a value in $[0,\infty]$) for \p-almost every  
curve $\ga$. For proofs of these and all other facts in this section 
we refer to Bj\"orn--Bj\"orn~\cite{BBbook} and
Heinonen--Koskela--Shanmugalingam--Tyson~\cite{HKSTbook}.
(Some of the references we mention below may not 
provide a proof in the generality considered here, but 
such proofs are given in \cite{BBbook}.)

The \p-weak upper gradients were introduced in
Koskela--MacManus~\cite{KoMc}. It was also shown there
that if $g \in \Lploc(X)$ is a \p-weak upper gradient of $f$,
then one can find a sequence $\{g_j\}_{j=1}^\infty$
of upper gradients of $f$ such that $g_j-g \to 0$ in $L^p(X)$.
If $f$ has an upper gradient in $\Lploc(X)$, then
it has a \emph{minimal \p-weak upper gradient} $g_f \in \Lploc(X)$
in the sense that for every \p-weak upper gradient $g \in \Lploc(X)$ of $f$ we have
$g_f \le g$ a.e., see Shan\-mu\-ga\-lin\-gam~\cite{Sh-harm}
and Haj\l asz~\cite{Haj03}. The minimal \p-weak upper gradient is well defined
up to a set of measure zero in the cone of nonnegative functions in $\Lploc(X)$.
Following Shanmugalingam~\cite{Sh-rev}, 
we define a version of Sobolev spaces on the metric space $X$.

\begin{deff} \label{deff-Np}
        Whenever $f\in L^p(X)$, let
\[
        \|f\|_{\Np(X)} = \biggl( \int_X |f|^p \, \dmu 
                + \inf_g  \int_X g^p \, \dmu \biggr)^{1/p},
\]
where the infimum is taken over all upper gradients of $f$.
The \emph{Newtonian space} on $X$ is 
\[
        \Np (X) = \{f: \|f\|_{\Np(X)} <\infty \}.
\]
\end{deff}

The space $\Np(X)/{\sim}$, where  $f \sim h$ if and only if $\|f-h\|_{\Np(X)}=0$,
is a Banach space and a lattice, see Shan\-mu\-ga\-lin\-gam~\cite{Sh-rev}.
In this paper we assume that functions are defined everywhere,
not just up to an equivalence class in the corresponding function space.
When we say that $f \in \Np(X)$ we thus assume that $f$
is a function defined everywhere. We say  that $f \in \Nploc(\Om)$ if
for every $x \in \Om$ there exists $r_x$ such that 
$B(x,r_x)\subset\Om$ and $f \in \Np(B(x,r_x))$.

If $f,h \in \Nploc(X)$, then $g_f=g_h$ a.e.\ in $\{x \in X : f(x)=h(x)\}$,
in particular $g_{\min\{f,c\}}=g_f \chi_{f < c}$ for $c \in \R$.
For these and other facts on \p-weak upper gradients, see, e.g.,
Bj\"orn--Bj\"orn~\cite{BBpreprint}, Section~3.

\begin{deff}
Let $\Om\subset X$ be an open set.
The (Sobolev) \emph{capacity} (with respect to $\Om$)
of a set $E \subset \Om$  is the number 
\begin{equation*} 
   \Cp(E;\Om) =\inf_u    \|u\|_{\Np(\Om)}^p,
\end{equation*}
where the infimum is taken over all $u\in \Np (\Om) $ such that $u=1$ on $E$.
\end{deff}

Observe that the above capacity is \emph{not} the so-called variational
capacity, which is obtained by minimizing $\|g_u\|_{L^p(\Om)}^p$ over $u\in \Np_0 (\Om) $ 
such that $u=1$ on $E$.

For a given set $E$ we will consider the capacity taken with respect
to different sets $\Om$. When the capacity is taken with respect to the underlying
metric space  $X$, we usually drop $X$ from the notation and merely write $\Cp(E)$.
The capacity is countably subadditive. For this and other properties as well as 
equivalent definitions of the capacity 
we refer to Bj\"orn--Bj\"orn~\cite{BBbook}.

We say that a property holds \emph{quasieverywhere} (q.e.)\ 
if the set of points  for which the property does not hold 
has capacity zero. When needed, we shall specify the capacity with 
respect to which q.e.\ is taken. The capacity is the correct gauge 
for distinguishing between two Newtonian functions. 
If $u \in \Np(X)$, then $u \sim v$ if and only if $u=v$ q.e.
Moreover, Corollary~3.3 in Shan\-mu\-ga\-lin\-gam~\cite{Sh-rev} shows that 
if $u,v \in \Np(X)$ and $u= v$ a.e., then $u=v$ q.e.

We now introduce the space of Newtonian functions with 
zero boundary values  as follows: 
\[
\Np_0(\Om;A)=\{f|_{\Om} : f \in \Np(A) \text{ and }
        f=0 \text{ in } A \setm \Om\},
\]
where $A\subset X$ is a measurable set containing $\Om$.
(In Section~\ref{sect-NpOmm}, this definition will be applied also to $\Om$
equipped with the Mazurkiewicz distance $d_M$, and then $A$ will be 
replaced by the Mazurkiewicz closure $\clOmm$ of $\Om$ with respect to
this metric.) As with the capacity, when $A=X$
we usually drop $X$ from the notation and merely write $\Np_0(\Om)$.
It is fairly easy to see that $\Np_0(\Om)=\Np_0(\Om;\overline{\Om})$,
see Bj\"orn--Bj\"orn~\cite{BBbook}.
One can also replace the assumption ``$f=0$ on $A\setm \Om$''
with ``$f=0$ q.e.\ on $A \setm \Om$''
without changing the obtained space $\Np_0(\Om; A)$.
Functions from $\Np_0(\Om; A)$ can be defined to be zero
q.e.\ in $A\setm \Om$ and we will regard them in that sense if needed.
Here q.e.\ is taken with respect to the ambient set $A$.

We say that $\mu$  is \emph{doubling} if 
there exists a \emph{doubling constant} $C>0$ such that for all balls 
$B=B(x_0,r):=\{x\in X: d(x,x_0)<r\}$ in~$X$,
\begin{equation*}
        0 < \mu(2B) \le C \mu(B) < \infty,
\end{equation*}
where $\lambda B=B(x_0,\lambda r)$.  

\begin{deff} \label{def.PI.}
We say that $X$ supports a \emph{\p-Poincar\'e inequality} if
there exist constants $C>0$ and $\lambda \ge 1$
such that for all balls $B \subset X$, 
all integrable functions $f$ on $X$ and all upper gradients $g$ of $f$, 
\begin{equation} \label{PI-ineq}
        \vint_{B} |f-f_B| \,\dmu
        \le C (\diam B) \biggl( \vint_{\lambda B} g^{p} \,\dmu \biggr)^{1/p},
\end{equation}
where $ f_B 
 :=\vint_B f \,\dmu 
:= \int_B f\, d\mu/\mu(B)$.
\end{deff}

In the definition of Poincar\'e inequality we can equivalently assume
that $g$ is a \p-weak upper gradient---see the comments above.

If $X$ is complete and  supports a \p-Poincar\'e inequality
and $\mu$ is doubling, then Lipschitz functions
are dense in $\Np(X)$, see Shan\-mu\-ga\-lin\-gam~\cite{Sh-rev}, and the functions
in $\Np(X)$
and those in $\Np(\Om)$ are \emph{quasicontinuous}
(see Theorem~\ref{thm-quasicont} below), i.e.\ 
for every $\eps>0$ there is an open set $U$ such that
$\Cp(U)<\eps$ and $f|_{X \setm U}$ is real-valued continuous.
This means that in the Euclidean setting, $\Np(\R^n)$ is the 
refined Sobolev space as defined in
Heinonen--Kilpel\"ainen--Martio~\cite[p.~96]{HeKiMa}; we refer interested readers
to~\cite{Sh-rev} for this fact, or to  Bj\"orn--Bj\"orn~\cite{BBbook} 
for a proof of this fact valid in weighted $\R^n$.
This is the main reason why, unlike in the classical Euclidean setting, we do not
require the functions $u$ in the definition of capacity to be $1$ in a 
neighbourhood of $E$. Moreover,  $X$ is \emph{quasiconvex}, i.e.\
there is a constant $L$ such that any two points
$x,y \in X$ can be connected by a curve of length
at most $Ld(x,y)$. This fact was first observed by Semmes.
For a proof see Haj\l asz--Koskela~\cite[Proposition~4.4]{HaKo}.
Recall also that $X$ is \emph{proper} if all closed bounded subsets of 
$X$ are compact. If $\mu$ is doubling then $X$ is complete if and only if $X$ is proper.
 
We will need the following results from 
Bj\"orn--Bj\"orn--Shan\-mu\-ga\-lin\-gam~\cite{BBS5}.

\begin{thm} \label{thm-quasicont}
\textup{(\cite[Theorem~1.1 and Corollary~1.3]{BBS5}
or \cite[Theorems~5.29 and 5.31]{BBbook})}
Assume that $X$ is proper  and that 
continuous functions are dense in $\Np(X)$.
\textup{(}This happens if, for example, $X$ is complete with $\mu$ doubling 
and supporting a \p-Poincar\'e inequality.\textup{)} Then 
\begin{enumerate}
\item \label{Cp-outer}
$\Cp$ is an \emph{outer capacity}, i.e.\  for all $E\subset \Om$,
\[
	\Cp(E)
	=\inf_{\substack{ G \supset E \\  G 
		\text{ open}}} 
	\Cp(G);
\]
\item \label{qcont}
every $u\in\Nploc(\Om)$ is quasicontinuous in\/ $\Om$.
\end{enumerate}
\end{thm}

We do not know if there is any metric measure space $X$ for which 
there is  a nonquasicontinuous 
function in $\Np(X)$, nor if there is any $X$ such that $\Cp$ is
not an outer capacity. However, even if we do not know that continuous 
functions form a dense subclass of $\Np(X)$,  
in proper spaces we have the outer capacity property at the level of capacitary 
null sets, as the following proposition shows.

\begin{prop}    \label{prop-zero-cap-capacitable}
\textup{(\cite[Proposition~1.4]{BBS5}
(or  \cite[Proposition~5.27]{BBbook}))}
Let $X$ be proper and let $E\subset X$ with $\Cp(E)=0$.
Then for every $\eps>0$, there is an open set $U\supset E$ with $\Cp(U)<\eps$.
\end{prop}

\section{The capacity \texorpdfstring{$\bCp(\,\cdot\,;\Om)$}{}}
\label{sect-newcap}

In this section we introduce a new capacity $\bCp(\,\cdot\,;\Om)$, which 
is useful in the study of Perron solutions later in the paper. 

\begin{deff} \label{deff-bCp}
For $E \subset \overline{\Om}$  let
\[
     \bCp(E;\Om)= \inf_{u \in \A_E} \|u\|_{\Np(\Om)}^p,
\]
where $u \in \A_E$ if $u \in \Np(\Om)$ satisfies
both $u \ge 1$ on $E \cap \Om$ and 
\[
      \liminf_{\Om \ni y \to x} u(y) \ge 1
	\quad \text{for all } x \in E \cap \bdy \Om.
\]
\end{deff}

For $E \subset \Om$ the new capacity $\bCp(E;\Om)$ equals $\Cp(E;\Om)$.
The novelty here is that we extend the ``$\Om$-capacity'' to sets in 
the closure $\overline{\Om}$. On the closure one may of course consider the 
capacity $\Cp(\,\cdot\,;\overline{\Om})$, but the new
capacity, being smaller (see the appendix), 
makes some of our results more general.
See the appendix for a comparison of various
capacities and Section~\ref{sect-examples} for
examples where we obtain better results using the new capacity.

Let us deduce some of the properties of the new capacity  $\bCp$.
The properties of subadditivity and outer capacity will be important.
By truncation it is easy to see that one may as well
take the infimum over all $u \in \At_E:=\{u \in \A_E : 0 \le u \le 1 \}$. 

\begin{prop} \label{prop-Cp} 
Let $E,E_1, E_2,\ldots$ be arbitrary subsets of\/ $\overline{\Om}$.
Then
\begin{enumerate}
\renewcommand{\theenumi}{\textup{(\roman{enumi})}}%
\item \label{Cp-emptyset}
  $\bCp(\emptyset;\Om)=0$\textup{;}
\item \label{Cp-meas}
  $\mu(E \cap \Om) \le \bCp(E;\Om)$\textup{;}
\item \label{Cp-subset}
  if $E_1 \subset E_2$, then $\bCp(E_1;\Om) \le \bCp(E_2;\Om)$\textup{;}
\item \label{Cp-subadd}
  $\bCp(\,\cdot\,;\Om)$ is countably subadditive, i.e.
  \[
      \bCp\biggl(\bigcup_{i=1}^\infty E_i;\Om\biggr) 
          \le \sum_{i=1}^\infty \bCp(E_i;\Om).
  \]
\end{enumerate}
\end{prop}

\begin{proof}
The claims \ref{Cp-emptyset}--\ref{Cp-subset} are immediate from the definition.
We now prove~\ref{Cp-subadd}.

Let $\eps>0$, and choose $u_i$ with $u_i \in \At_{E_i}$
and  upper gradients $g_i$ in $\Om$ such that
\[
    \|u_i\|_{L^p(\Om)}^p+ \|g_i\|_{L^p(\Om)}^p \le \bCp(E_i;\Om)+ \frac{\eps}{2^i}.
\]
Let $u=\sup_i u_i$ and $g=\sup_i g_i$.
It is an easy exercise to see that $g$ is an upper gradient of $u$,
see Lemma~1.28 in Bj\"orn--Bj\"orn~\cite{BBbook} for a proof.
Clearly $u \in \At_E$, where $E=\bigcup_{i=1}^\infty E_i$. Hence
\[
      \bCp(E;\Om)
           \le \|u\|_{\Np(\Om)}^p
       \le \int_\Om \sum_{i=1}^\infty u_i^p \,d\mu 
              + \int_\Om \sum_{i=1}^\infty g_i^p \,d\mu   
\le \sum_{i=1}^\infty \Bigl( \bCp(E_i;\Om)+ \frac{\eps}{2^i}\Bigr).
\]
Letting $\eps \to 0$ completes the proof of \ref{Cp-subadd}.
\end{proof}

\begin{prop}\label{prop-outercap}
Assume that all functions in $\Np(\Om)$ are quasicontinuous.
Then $\bCp(\,\cdot\,;\Om)$ is an outer capacity,
i.e.\  for all $E\subset\clOm$,
\[
	\bCp(E;\Om)
	=\inf_{\substack{ G \supset E \\  G 
		\text{ relatively open in\/ }\overline{\Om}}} 
	\bCp(G;\Om).
\]
\end{prop}

By Theorem~\ref{thm-quasicont}, we know that
Proposition~\ref{prop-outercap} applies 
when $X$ is complete and the measure on $X$ 
is doubling and supports a \p-Poincar\'e inequality.
In Section~\ref{sect-bdyM-Om}
we will apply this result to the $\bCp(\,\cdot\,;\Omm)$ capacity
defined below; this is possible since $\Np(\Omm)=\Np(\Om)$.

\begin{proof}
The fact that the left-hand side is not larger than the right-hand side
follows directly from the monotonicity in Proposition~\ref{prop-Cp}\,\ref{Cp-subset}.

To prove the converse inequality, let $E \subset \overline{\Om}$, 
$0<\eps<1$, and $u \in \At_E$ be such that 
\[
     \|u\|_{\Np(\Om)} \le \bCp(E;\Om)^{1/p}+\eps.
\]
By assumption, $u$ is quasicontinuous in $\Om$. Hence there is an open set 
$V \subset \Om$ with $\Cp(V;\Om)^{1/p} < \eps$ such that $u|_{\Om \setm V}$ is 
continuous. Thus, there is an open set $U \subset \Om$ such that
\[
          U \setm V = \{x \in \Om : u(x) > 1-\eps\} \setm V 
	\supset (E \cap \Om) \setm V.
\]
We can also find $v \ge \chi_V$ with $\|v\|_{\Np(\Om)} < \eps$. Let
\[
          w= \frac{u}{1-\eps} + v.
\]
Then $w \ge 1$ on $(U \setm V) \cup V = U \cup V$,
an open set containing $E \cap \Om$.
Moreover, for each $x \in E \cap \bdy \Om$ there is $r_x>0$ such that
\[
	u > 1-\eps \quad
		\text{in } B(x,r_x) \cap \Om,
\]
and hence $w \ge 1$ in $B(x,r_x) \cap \Om$. Therefore
\[
       W=U \cup V \cup 
                \bigcup_{x \in E \cap \bdy \Om} (B(x,r_x) \cap \overline{\Om}).
\]
is a relatively open subset of $\overline{\Om}$ containing $E$
and $w \in \A_W$. Hence
\begin{align*}
   \bCp(E;\Om)^{1/p} 
          & \le \inf_{\substack{ G \supset E \\ 
	G 		\text{ relatively open in }\overline{\Om}}} 
	 \bCp(G;\Om)^{1/p}
            \le \bCp(W;\Om)^{1/p} 
            \le \|w\|_{\Np(\Om)} \\
          &  \le \frac{1}{1-\eps} \|u\|_{\Np(\Om)} + \|v\|_{\Np(\Om)}
            \le \frac{1}{1-\eps} (\bCp(E;\Om)^{1/p}+\eps) + \eps.
\end{align*}
Letting $\eps \to 0$ completes the proof.
\end{proof}

For the sake of clarity we make the following explicit definition.
We set $\eR:=[-\infty,\infty]$.

\begin{deff} \label{deff-bCp-qcont}
A function $f \in \overline{\Om} \to \eR$ is 
\emph{$\bCp(\,\cdot\,;\Om)$-quasi\-con\-tin\-u\-ous}
if for every $\eps>0$ there is a relatively open set 
$U \subset \overline{\Om}$ such that
$\bCp(U;\Om)<\eps$ and $f|_{\overline{\Om} \setm U}$ 
is real-valued continuous.
\end{deff}

\emph{We assume from now on that $X$ is a complete
metric space supporting a \p-Poincar\'e inequality,
that $\mu$ is doubling, and that\/ $1<p<\infty$.}
It follows that $X$ is quasiconvex,
and in particular connected and locally connected.

\section{The Mazurkiewicz distance \texorpdfstring{$\dM$}{}}
\label{sect-dM}

\emph{In addition to the standing assumptions mentioned at the
end of the previous section, we
assume in this section that $\Om$ is a \emph{bounded domain}, i.e.\ a 
bounded nonempty open connected set.}

\medskip

\begin{deff}
We define the \emph{Mazurkiewicz distance} $\dM$ on $\Om$ by
\[
     \dM(x,y) =\inf_E \diam E,
\]
where the infimum is over all connected sets $E \subset \Om$
containing $x,y \in \Om$. We further define the \emph{inner metric} 
$\din$ on $\Om$ by
\[
     \din(x,y)=\inf_\ga l_\ga,
\]
where the infimum is taken over all curves $\ga:[0,l_\ga] \to \Om$ parameterized 
by arc length and such that $\ga(0)=x$ and $\ga(l_\ga)=y$.
\end{deff}

A consequence of the quasiconvexity of $X$ is that
each pair of points $x,y \in \Om$ can be  connected by a rectifiable curve
in $\Om$ (see Bj\"orn--Bj\"orn~\cite[Lemma~4.38]{BBbook}),
and so $\din(x,y)<\infty$. Hence both  $\din$ and $\dM$ are metrics on $\Om$.

\begin{remark} \label{rmk-M-dist}
The Mazurkiewicz distance was introduced
by Mazurkiewicz~\cite{mazurkiewicz16},
 in relation to a classification of points on $n$-dimensional Euclidean continua.
It  goes under different names in the literature,
and  is e.g.\  denoted $\rho_A$ in \cite{mazurkiewicz16}, 
called relative distance and denoted $\varrho_r$ in
Kuratowski~\cite{kuratowski2}, 
called Mazurkiewicz intrinsic metric and denoted $\delta_D$ 
in Karmazin~\cite{karmazin2008}, and called inner diameter distance 
in Aikawa--Hirata~\cite{AikawaHirata}, 
Freeman--Herron~\cite{FreemanHerron} 
and Herron--Sullivan~\cite{HerronSullivan}. Here and in
Adamowicz--Bj\"orn--Bj\"orn--Shan\-mu\-ga\-lin\-gam~\cite{ABBSprime}
and Bj\"orn--Bj\"orn--Shanmugalingam~\cite{BBStop}
we call it the Mazurkiewicz distance.
\end{remark}

\begin{lem} \label{lem-d-dm-din}
We always have $d \le \dM \le \din$. Furthermore, if\/
$\Om$ is $L$-quasi\-convex, then we also have
\(
   \din \le L d.
\)
\end{lem}

Note that even though $X$ is quasiconvex,
we will consider these distances with respect to $\Om$,
which, in general, is not quasiconvex.

\begin{proof}
The first inequality is obvious. As for the second inequality, 
let $\ga:[0,l_\ga] \to \Om$ be a curve, parameterized by
arc length, such that $\ga(0)=x$ and $\ga(l_\ga)=y$.
Then the image $\gahat:=\ga([0,l_\ga])$ is connected
and $\diam \gahat \le l_\ga$. Hence
\[
      d(x,y) \le \dM(x,y)=\inf_E \diam E
        \le \inf_\ga \diam \gahat 
        \le \inf_\ga l_\ga
        = \din(x,y).         
\]
If $\Om$ is $L$-quasiconvex, then $l_\ga\le Ld(x,y)$
for some curve $\ga$ in the infimum, proving the third inequality.
 \end{proof}

\begin{lem} \label{lem-arc-length}
For a curve $\ga:[0,l_\ga] \to \Om$,
arc lengths with respect to $d$, $\dM$ and $\din$ are the same.
\end{lem}

\begin{proof}
That arc lengths are the same with respect to $d$ and $\din$ is folklore,
for a proof see Bj\"orn--Bj\"orn~\cite[Lemma~4.43]{BBbook}.
It then follows from Lemma~\ref{lem-d-dm-din} that
arc length is also the same with respect to $\dM$.
\end{proof}

In the rest of the paper it will be important for us to
work with both the given metric $d$ and 
the Mazurkiewicz distance $\dM$. 
The inner metric $\din$ will however not be used in the rest
of the paper, apart from in some examples
in Section~\ref{sect-examples}.

In this paper, by $\clOmm$ we mean the completion of the metric space
$\Omm:=(\Om,\dM)$, where the Mazurkiewicz
distance $\dM$ comes from $\Om$.
On $\clOmm$, $\dM$ always refers to the metric
in $\clOmm$ inherited from the metric $\dM$ on $\Om$.
The focus of this paper is to use the Perron method  
to study solutions of Dirichlet problems with various boundary data
with respect to the Mazurkiewicz boundary. For this method to work 
it will be vitally important that $\clOmm$ is compact. 
It turns out that the compactness of $\clOmm$ 
has a very geometric characterization.
We state it next, before defining the concepts involved.

\begin{thm} \label{thm-clOmm-cpt}
The  closure\/ $\clOmm$ is compact if and only if\/
$\Om$ is finitely connected at the boundary.
\end{thm}

This theorem holds whenever $X$ is proper and locally connected. 
For a proof see Karmazin~\cite[Theorem~1.3.8]{karmazin2008}
(in Russian) or Bj\"orn--Bj\"orn--Shan\-mu\-ga\-lin\-gam~\cite{BBStop}.

\begin{deff}
We say that $\Om$ is \emph{finitely connected} at $x_0 \in \bdy \Om$ if
for every $r>0$ there is an open (in $X$) set $G$  
such that $x_0 \in G \subset B(x_0,r)$
and $G \cap \Om$ has only finitely many components.

If there is $N>0$ such that for every $r>0$
there is an open (in $X$) set $G$
such that $x_0 \in G \subset B(x_0,r)$
and such that $G \cap \Om$ has at most $N$ components,
then we say that $\Om$ is \emph{boundedly connected} at $x_0$.
If moreover $N$ is minimal, we say that 
$\Om$ is 
\emph{$N$-connected} at $x_0$. We say that
$\Om$ is \emph{locally connected} at $x_0 \in \bdy \Om$
if it is $1$-connected at $x_0$.

We say that $\Om$ has one of the above properties
\emph{at the boundary} if it has that property  at each boundary point.
\end{deff}

The terminology above follows N\"akki~\cite{nakki70}.
(N\"akki~\cite{nakki-private} has informed us that
he learned about the terminology from V\"ais\"al\"a, who
however first seems to have used it in print in~\cite{vaisala}.) 
For planar domains, the concept of finite connectedness at the boundary was 
used by Newman~\cite{newman} (only in the first edition of his book).
Beware that the notion of finitely connected domains is a 
completely different notion; a domain is finitely connected if its fundamental
group is finitely generated.
Observe also that the balls in the definition above are taken 
with respect to the given metric $d$.

If $\Om$ is finitely connected at the boundary, then
the map $\Phi : \clOmm \to \overline{\Om}$ defined below
sheds more light on the relation between the Mazurkiewicz
boundary $\bdym \Om$ and the metric boundary $\bdy \Om$.
On $\Omm$ the map $\Phi$ is the natural map given by
$\Phi(x)=x$ when $x\in\Omm$. This map is a $1$-Lipschitz
map on $\Omm$ (since $d \le \dM$), and hence 
has a unique continuous extension to $\clOmm$,  
which we again denote by $\Phi$; this extension is also $1$-Lipschitz.
If $x_0 \in \bdy \Om$ and $\Om$ is $N$-connected at $x_0$, then 
$\Phi^{-1}(x_0)$ consists of exactly $N$ points,
while if $\Om$ is finitely but not boundedly connected
at $x_0 \in \bdy \Om$, then $\Phi^{-1}(x_0)$ is an infinite countable set.
For a more detailed description,
see Bj\"orn--Bj\"orn--Shan\-mu\-ga\-lin\-gam~\cite{BBStop}.

\section{The Sobolev spaces \texorpdfstring{$\Np(\Omm)$}{}, 
\texorpdfstring{$\Np(\clOmm)$}{}}
\label{sect-NpOmm}

\emph{In addition to the standing assumptions described at the end of
Section~\ref{sect-newcap}, we assume in this section that\/ $\Om$ is a bounded domain
which is finitely connected at the boundary.}

\medskip

There are now two different metrics on $\Om$ of interest
here: the given metric $d$ and the Mazurkiewicz distance $\dM$. 
To make the distinction clear we denote the metric
space $(\Om,d)$ by $\Om$, and the metric space $(\Om,\dM)$ by
$\Omm$. We equip both of them with the measure $\mu$ 
(or strictly speaking, the restriction $\mu|_\Om$).
As sets, $\Om=\Omm$, and it is only the metrics that are different.
It is when taking closures that we really have to distinguish between
$\Om$ and $\Omm$. Note that $\Om \subset X$ and $\Omm \subset \clOmm$.
The closure $\overline{\Om}$ of $\Om$ in $X$ is compact because $X$ is proper. 
The closure of $\Omm$ is the completion $\clOmm$ as introduced in 
Section~\ref{sect-dM}, which is compact by Theorem~\ref{thm-clOmm-cpt}.

We also equip $\overline{\Om}$ and $\clOmm$ with measures as follows.
For $\overline{\Om}$ the natural choice is to equip it with $\mu|_{\overline{\Om}}$.
However, it is sometimes preferable to equip $\overline{\Om}$
with $\mu_0:=\mu|_{\Om}$ so that $\mu_0(\bdy \Om)=0$
(strictly speaking we let $\mu_0(E)=\mu(E \cap \Om)$
for $E \subset \overline{\Om}$). For $\clOmm$ we have no natural measure on 
$\bdy \Om$ and we always equip it with (the zero extension of) $\mu|_\Om$ 
so that $\mu(\bdym \Om)=0$. Since $\Om$ is an open subset of 
$\clOmm$, this zero extension is a Borel measure.

The two metrics $d$ and $\dM$ are locally equivalent in $\Om$, and give 
the same (not only equivalent) arc
lengths of curves. This is important from the point of view of 
upper gradients and Newtonian spaces, see Lemma~\ref{lem-arc-length}.

Let $f: \Om \to \eR$ be an arbitrary (extended real-valued)
function on $\Om$, and let $g:\Om \to [0,\infty]$.
As arc lengths for curves in $\Om$ are the same with respect to 
$d$ and $\dM$, $g$ will be an upper gradient of
$f$ with respect to $\Om$ if and only if it is an 
upper gradient with respect to $\Omm$. Since
we equip both metric spaces with the same measure
we see that $\Np(\Om)=\Np(\Omm)$ and
$\Nploc(\Om)=\Nploc(\Omm)$ (with the same (semi)norms).
It also follows that $g$ is a \p-weak upper gradient of
$f$ with respect to $\Om$ if and only if it is a 
\p-weak upper gradient with respect to $\Omm$.
See also Proposition~\ref{prop-Np0} where similar results are
obtained for Newtonian functions with zero boundary values. 

The following two lemmas relate Newtonian spaces and capacities
with respect to $\Om$, $\clOm$ and $\clOmm$.
Recall that $\Phi$ is the 1-Lipschitz extension to $\clOmm$ of
the identity map on $\Omm$, see the end of Section~\ref{sect-dM}.

\begin{lem} \label{lem-fPhi}
If $f \in \Np(\overline{\Om};\mu_0)$, then $\ft:=f \circ \Phi \in \Np(\clOmm)$.
Moreover,
\[ 
     \|\ft\|_{\Np(\clOmm)}= \|f\|_{\Np(\Om)}
          = \|f\|_{\Np(\overline{\Om};\mu_0)}.
\] 
\end{lem}

\begin{proof}
Let $g \in L^p(\overline{\Om})$ 
be an upper gradient of $f$ on $\overline{\Om}$
and let $\gt:=g \circ \Phi$. We shall show that $\gt$
is an upper gradient of $\ft$ on $\clOmm$.
To this end, let $\gat :[0,l_\gat] \to \clOmm$ be a curve,
parameterized by arc length, and $\ga = \Phi \circ \gat$.
As a composition of two $1$-Lipschitz functions, $\ga$ is $1$-Lipschitz
and it follows that it is a rectifiable curve, though it need not be
parameterized by arc length. Furthermore,
\[
     \Lip \ga(t):= \limsup_{h \to 0} 
         \frac{d(\ga(t+h),\ga(t))}{|h|}   \le 1.
\]
Hence
\begin{align*}
   |\ft(\gat(0))-\ft(\gat(l_\gat))|
	& =    |f(\ga(0))-f(\ga(l_\gat))| 
	 \le \int_{\ga} g \, ds \\
	& = \int_0^{l_\gat} g \circ \ga (t) \Lip \ga(t)\, dt 
	 \le \int_0^{l_\gat} \gt \circ \gat (t)  \, dt 
	 = \int_{\gat} \gt \, ds,
\end{align*}
and so $\gt \in L^p(\clOmm)$ is an upper gradient of $\ft$.
Hence $\ft \in \Np(\clOmm)$,  since 
\[ 
	\|\ft\|_{L^p(\clOmm)}
	= \|f\|_{L^p(\Om)}
	= \|f\|_{L^p(\overline{\Om};\mu_0)}.
\] 
Here  we have used the fact that $\mu(\bdym \Om)=\mu_0(\bdy \Om)=0$.

The restriction to $\Om$ of the minimal \p-weak upper gradient $g_f$ of
$f$ on $\overline{\Om}$ is minimal also as a \p-weak upper gradient on 
$\Om$, i.e.\ if we denote the minimal \p-weak upper gradient
of $f|_\Om$ by $g_{f|_\Om}$, then $g_f|_{\Om}=g_{f|_\Om}$ in $\Om$.
This follows from Shanmugalingam~\cite[Lemma~3.2]{Sh-harm}, see also
Bj\"orn--Bj\"orn~\cite[Lemma~2.23]{BBbook}. 
Similarly $g_{\ft}|_{\Om}=g_{\ft|_\Om}=g_{f|_\Om}$ in $\Om$. Hence 
\[ 
	\|g_{\ft}\|_{L^p(\clOmm)}
	= \|g_{f|_{\Om}}\|_{L^p(\Om)}
	= \|g_{f}\|_{L^p(\overline{\Om};\mu_0)}.
	\qedhere
\] 
\end{proof}

\begin{lem} \label{lem-Phicap-Omm}
Let $E \subset \overline{\Om}$. Then 
\[ 
	\bCp(\Phi^{-1}(E);\Omm)
	\le \bCp(E;\Om)
	\le \Cp(E).
\]
\end{lem}

The first inequality is actually an equality, see Proposition~\ref{prop-Phicap-Omm-equal}.
Further comparisons of these capacities will be given in the appendix. 

\begin{proof}
For the first inequality, take $u \in \A_E$ with respect to $\Np(\Om)$ and let 
$\ut=u \circ \Phi  : \Omm \to \eR$.  Let $x \in \Phi^{-1}(E)\cap\bdym\Om$ and take a 
sequence $y_j \in \Omm$ such that $y_j \dMto x$. 
Since $d \le \dM$, $y_j \to \Phi(x)$ in the given metric $d$.
Hence, because $\Phi(x) \in E\cap\bdry\Om$,
\[ 
     \liminf_{j \to \infty} u(y_j) \ge 1.
\]
Thus $\ut$ is admissible in the definition of $\bCp(\Phi^{-1}(E);\Omm)$.
Taking infimum over all such $u$,  together with the fact that $\Np(\Om)=\Np(\Omm)$ 
(with the same norms), shows the first inequality.

We now turn to the second inequality, 
which is obvious if $\Cp(E)=\infty$. Thus we may assume that $\Cp(E)<\infty$.
Let $\eps>0$. Since $\Cp$ is an outer capacity,
by Theorem~\ref{thm-quasicont}\,\ref{Cp-outer}
we can find an open set $G \supset E$ such that $\Cp(G) < \Cp(E) +\eps$.
Therefore we can find $u \in \Np(X)$ such that $u \ge 1$ in $G$
and $\|u\|_{\Np(X)}^p < \Cp(E) +\eps$. Since $G$ is open, we have 
$u \in \A_{E}$. Hence
\[
    \bCp(E;\Om) \le \|u\|_{\Np(\Om)}^p \le \Cp(E) +\eps.
\]
Letting $\eps \to 0$ concludes the proof.
\end{proof}

To discuss solutions of Dirichlet problems, we need to
compare boundary values of Newtonian functions. 
Two Newtonian functions have the same boundary values if their difference belongs 
to  $\Np_0(\Om):=\Np_0(\Om;X)=\Np_0(\Om;\overline{\Om})$,
where the last equality was pointed out in Section~\ref{sect-prelim}.
Similarly, we can define $\Np_0(\Omm):=\Np_0(\Omm;\clOmm)$
and $\Np_0(\Om;\mu_0):=\Np_0(\Om; (\overline{\Om},\mu_0))$. 
The capacity $\bCp(\,\cdot\,;\Omm)$, and the related quasicontinuity, 
are defined in a manner similar to $\bCp(\,\cdot\,;\Om)$, and
all the results in Section~\ref{sect-newcap} have direct counterparts
for $\bCp(\,\cdot\,;\Omm)$.

The following is the main result of this section. Note that in general,  as sets,
\[
   \Np(\clOm) \subsetneq \Np(\clOm;\mu_0) 
          \ne \Np(\clOmm),
\]
and so it is the requirement that $u=0$ outside $\Om$ which gives the 
following equality.

\begin{prop} \label{prop-Np0}
We have $\Np_0(\Om)=\Np_0(\Om;\mu_0) = \Np_0(\Omm)$. 

Furthermore, if $f$ is a function in this class, then $g$ is 
 a\/ \textup{(}\p-weak\/\textup{)} upper gradient of the 
 zero-extension of $f$ with respect to\/ $\overline{\Om}$
 if and only if $g\circ \Phi$ is a\/ \textup{(}\p-weak\/\textup{)} 
 upper gradient of $f$ with respect to\/ $\clOmm$. Here\/ 
 $\Phi$ is the extension of the canonical identity map\/ 
 $\Om\to\Omm$ to their respective closures.
\end{prop}

\begin{proof}  
The inclusion $\Np_0(\Om)\subset\Np_0(\Om;\mu_0)$ is clear.

Assume next that $f \in \Np_0(\Om;\mu_0)$. 
Then by definition $f\in \Np(\clOm;\mu_0)$ with $f=0$ on $\bdy \Om$.
Lemma~\ref{lem-fPhi} yields $f \circ \Phi \in \Np(\clOmm)$,
and as $f \circ \Phi \equiv 0$ on $\bdym \Om$ we see that
$f \circ \Phi \in \Np_0(\Omm)$. Since $f \equiv f \circ \Phi$ on $\Om=\Omm$, 
we conclude that  $\Np_0(\Om;\mu_0) \subset \Np_0(\Omm)$. 

Finally, if $f \in \Np_0(\Omm)$, then $f\in\Np(\clOmm)$ with $f=0$ on $\bdym \Om$.
Let $g \in L^p(\clOmm)$ be an upper gradient of $f$ on $\clOmm$, and set
\[
   \ft= \begin{cases}
	f & \text{in } \Om, \\
	0& \text{on } X \setm \Om,
	\end{cases}
	\quad \text{and} \quad
   \gt= \begin{cases}
	g & \text{in } \Om, \\
	0& \text{on } X \setm  \Om.
	\end{cases}
\]
We shall show that $\gt$ is an upper gradient of $\ft$ on 
$X$. Let $\ga :  [0,l_\ga] \to X$ be a rectifiable curve.
If $\ga([0,l_\ga]) \subset \Om$, then clearly
\[
     |\ft(\ga(0))-\ft(\ga(l_\ga))| 
     = |f(\ga(0))-f(\ga(l_\ga))| 
	\le \int_\ga g \, ds
	= \int_\ga \gt \, ds,
\]
as arc lengths are the same with respect to the  metrics $d$ and $\dM$.
If $\ga([0,l_\ga])\cap\Om$ is empty, then the above inequality is trivial.
We may therefore (by splitting $\ga$ into two parts and reversing
the orientation if necessary) assume that $\ga([0,l_\ga]) \not\subset \Om$ and that
$\ga(0)\in\Om$. Let $t_0=\inf\{0 \le t \le l_\ga : \ga(t) \notin \Om\}>0$
and $t_1=\sup\{0 \le t \le l_\ga : \ga(t) \notin \Om\}\le l_\ga$.
Because $\Om$ is open, we have $\ga(t_0),\ga(t_1) \in X \setm  \Om$.
We want to show that 
\[
     |\ft(\ga(0))-\ft(\ga(t_0))| \le \int_{\ga|_{[0,t_0)}} \gt \, ds.
\]
Defining the map $\gat:[0,t_0)\to\Omm$ by $\gat(t)=\ga(t)$ for $0 \le t < t_0$,
we note that $\ga$ is arc-length parameterized with respect
to $\dM$, by Lemma~\ref{lem-arc-length}.  It follows that
$\gat$ is a $1$-Lipschitz map from $[0,t_0)$ to $\Omm$, and hence has 
a continuous extension, also denoted $\gat$, from $[0,t_0]$ to $\clOmm$.   
If $\gat(t_0)\in\Omm=\Om$, then so would $\ga(t_0)$. Thus we conclude that 
$\gat(t_0)\in\partial_M\Om$, i.e.\ $\ft(\ga(t_0))=f(\gat(t_0))=0$. Hence
\[
     |\ft(\ga(0))-\ft(\ga(t_0))|
      =  |f(\gat(0))-f(\gat(t_0))|
      \le \int_{\gat|_{[0,t_0)}} g \, ds
      = \int_{\ga|_{[0,t_0)}} \gt \, ds,
\]
as $g$ is an upper gradient of $f$ on $\clOmm$.  Similarly, if $t_1<l_\ga$, then
\[
     |\ft(\ga(t_1))-\ft(\ga(l_\ga))| 
	\le  \int_{\ga|_{(t_1,l_\ga]}} \gt \, ds.
\]
The above inequality holds trivially if $t_1=l_\ga$. Hence
\begin{align*}
     |\ft(\ga(0))-\ft(\ga(l_\ga))| 
     & \le |\ft(\ga(0))-\ft(\ga(t_0))| 
      +|\ft(\ga(t_0))-\ft(\ga(t_1))|  \\
	& \quad     + |\ft(\ga(t_1))-\ft(\ga(l_\ga))| \\
	& \le \int_{\ga|_{[0,t_0)}} \gt \, ds
	+ 0 + \int_{\ga|_{(t_1,l_\ga]}} \gt \, ds \\
	&	\le \int_{\ga} \gt \, ds.
\end{align*}
Thus $\gt \in L^p(X)$ is an upper gradient of $f$ on $X$, and hence $f \in \Np_0(\Om)$.
\end{proof}

For the Dirichlet problem in this paper we will also need
the following consequence of Proposition~\ref{prop-Np0}.

\begin{prop} \label{prop-Np0-qcont-Omm} 
Let $f \in \Np_0(\Omm)$. Then $f$, extended by\/ $0$ to $\bdym \Om$,
is $\bCp(\,\cdot\,;\Omm)$-quasi\-con\-tin\-u\-ous.
\end{prop}

\begin{proof} 
Let 
\[
   \ft= \begin{cases}
	f & \text{in } \Om, \\
	0& \text{on } X \setm \Om.
	\end{cases}
\]
By  Proposition~\ref{prop-Np0}, $\ft\in\Np_0(\Om)$. 
So  $\ft \in \Np(X)$, and hence by Theorem~\ref{thm-quasicont}\,\ref{qcont}
 it is  quasicontinuous in $X$. For $\eps>0$ there exists an open set $U \subset X$ 
 with $\Cp(U)< \eps$ such that $\ft|_{X \setm U}$ is continuous.
Then $\Ut:= \Phi^{-1}(U)$ is open in $\clOmm$ by the continuity of $\Phi$.
Moreover, $f=\ft \circ \Phi$ and therefore
\[
    f|_{\clOmm \setm \Ut} = \ft|_{X \setm U} \circ \Phi
\]
is continuous. Lemma~\ref{lem-Phicap-Omm} shows that
$\bCp(\Ut;\Omm) \le \Cp(U) < \eps$.
As $\eps>0$ was arbitrary this shows that $f$ is 
$\bCp(\,\cdot\,;\Omm)$-quasi\-con\-tin\-u\-ous in $\clOmm$.
\end{proof}

\section{\texorpdfstring{\p-harmonic}{p-harmonic} and superharmonic functions}
\label{sect-superharm}

In this section we introduce \p-harmonic and superharmonic functions,
as well as obstacle problems, which all will be needed in later sections.
For further discussion and references on these topics  see Bj\"orn--Bj\"orn~\cite{BBbook}
(which also contains proofs of the facts mentioned in this section).

\begin{deff} \label{def-quasimin}
A function $u \in \Nploc(\Om)$ is a
\emph{\textup{(}super\/\textup{)}minimizer} in $\Om$
if 
\[ 
      \int_{\phi \ne 0} g^p_u \, d\mu
           \le \int_{\phi \ne 0} g_{u+\phi}^p \, d\mu
           \quad \text{for all (nonnegative) } \phi \in \Np_0(\Om).
\] 
A \emph{\p-harmonic function} is a continuous minimizer.
\end{deff}

For characterizations of minimizers and superminimizers  see A.~Bj\"orn~\cite{ABkellogg}.
Minimizers were first studied for functions in $\Np(X)$
in Shanmugalingam~\cite{Sh-harm}, and
it was shown in Kinnunen--Shanmugalingam~\cite{KiSh01} that 
under the standing assumptions of this paper,
minimizers can be modified on a set of zero capacity to obtain
a \p-harmonic function. For a superminimizer $u$,  it was shown by
Kinnunen--Martio~\cite{KiMa02} that its \emph{lower semicontinuous regularization}
\begin{equation}
\label{essliminf}
 u^*(x):=\essliminf_{y\to x} u(y)= \lim_{r \to 0} \essinf_{B(x,r)} u
\end{equation}
is also a superminimizer  and $u^*= u$ q.e. 

We follow Kinnunen--Martio~\cite{KiMa02} in the following
definition of the obstacle problem. Let $V \subset X$ be a
nonempty bounded open set with $\Cp(X \setm V)>0$. (If $X$ is
unbounded then the condition $\Cp(X \setm V)>0$ is of course immediately fulfilled.)

\begin{deff} \label{deff-obst}
For $f \in \Np(V)$ and $\psi : V \to \eR$, we set
\[
    \K_{\psi,f}(V)=\{v \in \Np(V) : v-f \in \Np_0(V)
            \text{ and } v \ge \psi \ \text{a.e. in } V\}.
\]
A function $u \in \K_{\psi,f}(V)$ is a \emph{solution of the $\K_{\psi,f}(V)$-obstacle problem}
if
\[
       \int_V g^p_u \, d\mu
       \le \int_V g^p_v \, d\mu
       \quad \text{for all } v \in \K_{\psi,f}(V).
\]
\end{deff}

A solution to the $\K_{\psi,f}(V)$-obstacle problem is easily
seen to be a superminizer in $V$. Conversely, a superminizer $u$ in $\Om$ is a solution
of the $\K_{u,u}(V)$-obstacle problem for all $V \Subset \Om$,
i.e. $V$ such that $\overline{V}$ is a compact subset of $\Om$. 

Kinnunen--Martio~\cite[Theorem~3.2]{KiMa02} 
showed that if $\K_{\psi,f}(V)$ is nonempty,  then
there is a solution $u$ of the $\K_{\psi,f}(V)$-obstacle problem,
and this solution is unique up to equivalence in $\Np(V)$. Moreover, $u^*$ 
is the unique lower semicontinuously regularized solution.
If the obstacle $\psi$ is continuous, then $u^*$
is also continuous, see~\cite[Theorem~5.5]{KiMa02}.
The obstacle $\psi$, as a continuous function, is even allowed
to take the value $-\infty$. Given $f \in \Np(V)$, we let $\oHpind{V} f$ denote the
continuous solution of the $\K_{-\infty,f}(V)$-obstacle problem;
this function is \p-harmonic in $V$ and takes on the same boundary values
(in the Sobolev sense) as $f$ on $\partial V$, and hence it is also
called the solution of the Dirichlet problem with Sobolev boundary values. 
When $f \in \Np(X)$ this solution agrees with the one
constructed in~\cite{Sh-harm} and studied in~\cite{KiSh01}.

\begin{deff} \label{deff-superharm}
A function $u : \Om \to (-\infty,\infty]$ is \emph{superharmonic}
in $\Om$ if
\begin{enumerate}
\renewcommand{\theenumi}{\textup{(\roman{enumi})}}%
\item $u$ is lower semicontinuous;
\item \label{cond-b}
$u$ is not identically $\infty$ in any component of $\Om$;
\item \label{cond-c}
for every nonempty open set $V \Subset \Om$ and all functions
$v \in \Lip(X)$, we have $\oHpind{V} v \le u$ in $V$
whenever $v \le u$ on $\bdy V$.
\end{enumerate}

A function $u : \Om \to [-\infty,\infty)$ is \emph{subharmonic}
in $\Om$ if $-u$ is superharmonic.
\end{deff}

This definition of superharmonicity is equivalent to  the ones in
Hei\-no\-nen--Kil\-pe\-l\"ai\-nen--Martio~\cite{HeKiMa} and
Kinnunen--Martio~\cite{KiMa02}, see Theorem~6.1 in A.~Bj\"orn~\cite{ABsuper}.
A locally bounded superharmonic function is a superminimizer,
and all superharmonic functions are lower semicontinuously regularized.
Conversely, any lower semicontinuously regularized superminimizer
is superharmonic.

By Proposition~\ref{prop-Np0}, functions that are
\p-harmonic on an open subset of $\Om$ with respect to
either of the metrics $d$ and $\dM$ will be \p-harmonic
on that subset with respect to both metrics. 
By Proposition~\ref{prop-Np0} we also see that if $f\in \Np(\Om)=\Np(\Omm)$,
then $\K_{\psi,f}(\Om)=\K_{\psi,f}(\Omm)$, and thus the 
obstacle problem is exactly the same for both metrics.
In particular, $\oHpind{\Om} f=\oHpind{\Omm} f$ for $f \in \Np(\Om)$.

If we let $V \Subset \Om$ and equip it with the 
Mazurkiewicz  distance $\dM$ and the measure $\mu$, 
both inherited from $\Om$, we similarly see that 
$\K_{\psi,f}(V)=\K_{\psi,f}(V;\dM)$ since the metrics are equivalent on 
$\overline{V}$ and arc lengths are the same. It follows 
that also the class of all superharmonic functions on $\Om$ is the same 
with respect to both metrics $d$ and $\dM$. 
Thus within $\Om$ we have no reason to distinguish between, e.g.,
\p-harmonic functions defined using the metric $d$ and the metric $\dM$.

\section{Perron solutions with respect to \texorpdfstring{$\Omm$}{}}
\label{sect-Perron-Omm}

\emph{In addition to the standing assumptions 
described at the end of Section~\ref{sect-newcap}, 
we assume in this section that\/ $\Om$ is a 
bounded domain which is finitely connected at the boundary
and that, as in Section~\ref{sect-superharm}, $\Cp(X \setm \Om)>0$.}

\medskip

The main point of this paper is  that in considering
the Dirichlet boundary value problem there
is a difference between  $\Om$ and $\Omm$.
However, we saw in the previous section that
the Sobolev solutions $\oHpind{\Om} f$ and $\oHpind{\Omm} f$
coincide for $f\in\Np(\Om)$. For this reason we will usually denote 
this common solution by $\oHp f$. We also write $\K_{\psi,f}:=\K_{\psi,f}(\Om)$.

We shall now consider the Dirichlet problem for arbitrary functions
defined on the Mazurkiewicz boundary $\bdym\Om$.
This will be done by means of Perron solutions on $\Omm$ defined below.
The distinction from Perron solutions on $\Om$ is subtle but has 
important consequences for the Dirichlet problem since the Mazurkiewicz
boundary $\bdym\Om$ is finer than $\bdy\Om$.

\begin{deff}   \label{def-Perron}
Given a function $f : \bdym \Om \to \eR$, let $\UU_f(\Omm)$ be the set of all 
superharmonic functions $u$ on $\Omm$, bounded from below,  such that 
\[ 
	\liminf_{\Om \ni y \dMto x} u(y) \ge f(x) \quad \text{for all }
	x \in \bdym \Om.
\] 
The \emph{upper Perron solution} of $f$ is the function
\[ 
    \uHpind{\Omm} f (x) = \inf_{u \in \UU_f(\Omm)}  u(x), \quad x \in \Om.
\]
Similarly, let $\LL_f(\Omm)$ be the set of all 
subharmonic functions $u$ on $\Om$, bounded from above, such that 
\[
\limsup_{\Om \ni y \dMto x} u(y) \le f(x) \quad \text{for all } x \in \bdym \Om,
\]
and define the \emph{lower Perron solution} of $f$ by
\[ 
    \lHpind{\Omm} f (x) = \sup_{u \in \LL_f(\Omm)}  u(x), \quad x \in \Om.
\]
If $\uHpind{\Omm} f = \lHpind{\Omm} f$, then we let $\Hpind{\Omm} f := \uHpind{\Omm} f$
and $f$ is said to be \emph{resolutive} with respect to $\Omm$.

We similarly define $\uHpind{\Om} f$, $\lHpind{\Om} f$ and
$\Hpind{\Om} f$ for $f : \bdy \Om \to \eR$.
\end{deff}

  Immediate consequences of the above definition are that 
$\lHpind{\Omm} f= - \uHpind{\Omm} (-f)$ and that
\begin{equation} \label{eq-f1<f2}
    \uHpind{\Omm}f_1\le \uHpind{\Omm}f_2, \quad \text{if } f_1\le f_2.
\end{equation}

Observe that $\uHpind{\Om} f$ is \p-harmonic unless it is identically $\pm \infty$,
see Theorem~4.1 in Bj\"orn--Bj\"orn--Shan\-mu\-ga\-lin\-gam~\cite{BBS2}.
The proof therein applies also to $\uHpind{\Omm}f$ without any change.
The following comparison principle makes it possible to compare
the upper and lower Perron solutions.

\begin{prop} \label{prop-comparison}
Assume  that $u$ is superharmonic and that $v$ is subharmonic in\/ $\Om$. If 
\begin{equation} \label{eq-comp}
       \infty \ne  \limsup_{\Om \ni y \dMto x} v(y)
        \le \liminf_{\Om \ni y \dMto x} u(y) \ne -\infty
        \quad \text{for all } x \in \bdym \Om,
\end{equation}
then $v \le u$ in\/ $\Om$. 
\end{prop}

\begin{cor} \label{cor-uHp-lHp}
If $f:\bdym\Om\to\R$, then  
\[
   \uHpind{\Omm} f \ge \lHpind{\Omm} f.
\]
\end{cor}

The result corresponding to Proposition~\ref{prop-comparison}
with respect to the given  metric $d$  was obtained in
Kinnunen--Martio~\cite[Theorem~7.2]{KiMa02}.

\begin{proof}[Proof of Proposition~\ref{prop-comparison}]
Let $\Om_1 \Subset \Om_2 \Subset \cdots \Subset \Om= \bigcup_{k=1}^\infty \Om_k$
and $\eps >0$. For every $x\in\bdym\Om$, it follows from \eqref{eq-comp} that
\[
        \liminf_{\Om \ni y \dMto x} {(u(y)-v(y)) \ge 0 }
\]
and hence there is a ball $B_x^M\ni x$ (with respect to the metric $\dM$ on $\clOmm$) such that 
\[
u-v> -\eps \quad \text{in }B_x^{M}\cap\Om.
\]
By the compactness of $\clOmm$ (recall that we assume $\Om$ to be finitely connected at the
boundary), there are finitely many balls $B_{x_1}^M,\cdots,B_{x_N}^M$ and some $k>1/\eps$ such that 
\[ 
    \clOmm \subset \Om_k \cup B_{x_1}^{M} \cup \cdots \cup B_{x_N}^{M}.
\]
It follows that $v < u+\eps$ on $\bdy \Om_k$.  An application of~\cite[Theorem~7.2]{KiMa02} 
to $u+\eps$ and $v$ in $\Om_k$ now tells
us that $v\le u+\eps$ on $\Om_k$. Letting $\eps\to 0$ completes the proof.
\end{proof}

The following is the main result of this section.
Its consequences will be given in Section~\ref{sect-bdyM-Om}.

\begin{thm} \label{thm-Newt-resolve-Omm}
Let $f: \clOmm \to \eR$ be a $\bCp(\,\cdot\,;\Omm)$-quasi\-con\-tin\-u\-ous function
such that $f|_{\Om} \in \Np(\Om)$. Then $f$ is resolutive with respect to\/ $\Omm$ 
and $\Hpind{\Omm} f = \oHp f$.
\end{thm}

The corresponding result for $\Om$ under the assumption that
$f \in \Np(X)$ (in which case the quasicontinuity of $f$ is automatic
by Theorem~\ref{thm-quasicont}\,\ref{qcont}) was obtained in
Bj\"orn--Bj\"orn--Shan\-mu\-ga\-lin\-gam~\cite[Theorem~5.1]{BBS2}. 
The proof here is more intricate since we need to be more careful 
with issues of quasicontinuity with respect to the capacity 
$\bCp$, which does not come for free.
We also need the following modification of Lemma~5.3 in
Bj\"orn--Bj\"orn--Shan\-mu\-ga\-lin\-gam~\cite{BBS2}.

\begin{lem} \label{lem-Newt}
Let\/ $\{U_k\}_{k=1}^\infty$ be a decreasing sequence of 
relatively open sets in\/ $\clOmm$ such that $\bCp(U_k;\Omm) <2^{-kp}$.
Then there exists a decreasing sequence of nonnegative 
functions\/ $\{\psi_j\}_{j=1}^\infty$ on\/ $\Om$ such that\/ $\| \psi_j \|_{\Np(\Om)} < 2^{-j}$ 
and $\psi_j \ge k-j$ in $U_k\cap\Om$.
\end{lem}

\begin{proof}
Let $\psi_j=\sum_{k=j+1}^\infty f_k$, where $f_k\in\A_{U_k}$ with 
$\|f_k\|_{\Np(\Om)}<2^{-k}$ are admissible in the definition of $\bCp(U_k;\Omm)$.
\end{proof}

To prove Theorem~\ref{thm-Newt-resolve-Omm} we will
also need the following proposition, which
summarizes some useful convergence results 
for obstacle and Dirichlet problems. It consists of special cases of 
Farnana~\cite[Theorem~3.3]{Fa3}  and Kinnunen--Marola--Martio~\cite[Theorem~3]{KiMaMa}, 
but can also be found in Bj\"orn--Bj\"orn~\cite{BBbook} as Proposition~10.18 and Corollary~10.20.
For $f,f_j\in\Np(X)$ these results are due to Kinnunen--Shanmuga\-lingam~\cite{KiSh-polar} 
and Shan\-mu\-ga\-lin\-gam~\cite{Sh-conv}. 

\begin{prop}\label{prop-3.2}
Let\/ $\{f_j\}_{j=1}^\infty$ be a q.e.\
decreasing sequence of functions in $\Np(\Om)$ such that 
$f_j\to f$ in $\Np(\Om)$ as $j \to\infty$.
Then $\oHp f_j$ decreases to $ \oHp f$ locally uniformly in\/ $\Om$.

Moreover, if $u$ and $u_j$ are solutions of the $\K_{f,f}$- and
$\K_{f_j,f_j}$-obstacle problems, $j=1,2,\ldots$\,, then\/
$\{u_j\}_{j=1}^\infty$  decreases q.e.\ in\/ $\Omega $ to $u$.
\end{prop}

\begin{proof}[Proof of Theorem~\ref{thm-Newt-resolve-Omm}]
Assume first that $f \ge 0$.
Extend $\oHp f$ to $\clOmm$ by letting $Hf:=f$ on $\bdym\Om$.
We first show that $\oHp f$ is $\bCp(\,\cdot\,;\Omm)$-quasi\-con\-tin\-u\-ous.
Let $h=f-\oHp f\in \Np_0(\Omm)$,  with $h \equiv 0$ on $\bdym \Om$. 
Proposition~\ref{prop-Np0-qcont-Omm} shows that 
$h$ is $\bCp(\,\cdot\,;\Omm)$-quasi\-con\-tin\-u\-ous. Thus,
$\oHp f=f+h$  is also $\bCp(\,\cdot\,;\Omm)$-quasi\-con\-tin\-u\-ous on 
$\clOmm$, by the subadditivity of the $\bCp(\,\cdot\,;\Omm)$-capacity, see 
Proposition~\ref{prop-Cp}. Hence we can find relatively open subsets 
$G_j \subset\clOmm$, $j=1,2,\ldots$, such that $\bCp(G_j;\Omm) < 2^{-jp}$
and such that $\oHp f|_{\clOmm \setm G_j}$ is continuous.
Let $U_k=\bigcup_{j=k+1}^\infty G_j$, $k=1,2,\ldots$.
Then $\{U_k\}_{k=1}^\infty$  is a decreasing sequence 
of relatively open subsets of $\clOmm$
such that $\bCp(U_k;\Omm) < 2^{-kp}$ 
and $\oHp f|_{\clOmm \setm U_k}$ is  continuous.
 
Consider the decreasing sequence
of nonnegative functions $\{\psi_j\}_{j=1}^\infty$ given by
Lemma~\ref{lem-Newt} with respect to this sequence of sets.
Let $f_j = \oHp f + \psi_j$ (which is only defined in $\Om$)
and let $\phi_j$ be the lower semicontinuously regularized solution 
of the $\K_{f_j,f_j}$-obstacle problem.

For positive intergers $m$, by Lemma~\ref{lem-Newt}, 
\begin{equation} \label{N1}
    f_j \ge \psi_j \ge m \quad \text{on }U_{m+j}\cap\Om.
\end{equation}
Let $\eps >0$ and $x \in \bdym \Om$.  If $x\notin U_{m+j}$, then
by the continuity of $\oHp f|_{\clOmm \setm U_{m+j}}$
there is a neighbourhood $V_x$ of $x$ in $\clOmm$ such that 
\begin{equation} \label{N2}
 f_j(y) \ge \oHp f(y) \ge \oHp f(x) - \eps = f(x) -\eps
 \quad \text{for } y \in (V_x\cap\Om) \setm U_{m+j}.
\end{equation}
Combining \eqref{N1} and \eqref{N2} we see that for 
$x\in\partial_M\Om\setminus U_{m+j}$,
\begin{equation}   \label{eq-fj-ge-m}
    f_j \ge \min\{f(x)-\eps,m\}
       \quad \text{in  }V_x\cap\Om.
\end{equation}
On the other hand, if $x\in U_{m+j}$, then setting $V_x=U_{m+j}$, we see
by~\eqref{N1} that~\eqref{eq-fj-ge-m} holds as well.
As a solution to the $\K_{f_j,f_j}$-obstacle problem,  $\phi_j$ is 
lower semicontinuously regularized and $\phi_j \ge f_j$ q.e.
It follows that $\phi_j(y) \ge \min\{f(x)-\eps,m\}$
for \emph{every} $y \in V_x \cap \Om$. Hence 
\[ 
   \liminf_{\Om \ni y \dMto x} \phi_j(y) \ge \min\{f(x)-\eps,m\}.
\]
Letting $\eps \to 0$ and $m \to \infty$, we see that
\[
   \liminf_{\Om \ni y \dMto x} \phi_j(y) \ge f(x) \quad \text{for all } x \in \bdym \Om.
\]
As $\phi_j$ is superharmonic, it follows that $\phi_j \in \UU_f(\Omm)$,
and hence that $\phi_j \ge \uHpind{\Omm} f$.

Since $\oHp f$ clearly is a solution of the $\K_{\oHp f,\oHp f}$-obstacle problem,
we see by Proposition~\ref{prop-3.2} that $\{\phi_{j}\}_{j=1}^\infty$ 
decreases q.e.\ to $\oHp f$. Hence $\uHpind{\Omm} f \le \oHp f$ q.e.\ in $\Om$.

Next, let $f \in \Np(\Om)$ be arbitrary. Then by~\eqref{eq-f1<f2}, 
Proposition~\ref{prop-3.2}, and the above argument,
\[
      \uHpind{\Omm} f  \le \lim_{m \to -\infty} \uHpind{\Omm} \max\{f,m\}
          \le \lim_{m \to -\infty} \oHp \max\{f,m\}
          = \oHp f  \quad \text{q.e.\ in }\Om.
\]
Since both $\uHpind{\Omm} f$ and $\oHp f$ are continuous, 
we have $\uHpind{\Omm} f \le \oHp f$ everywhere in $\Om$.
It then follows from Corollary~\ref{cor-uHp-lHp} that 
\[
   \lHpind{\Omm} f = - \uHpind{\Omm} (-f) \ge - \oHp (-f) = \oHp f
       \ge \uHpind{\Omm} f \ge \lHpind{\Omm} f,
\] 
and hence that $\oHp f = \lHpind{\Omm} f = \uHpind{\Omm} f$.
\end{proof}

\begin{remark} \label{rmk-prime-end}
It is shown in Adamowicz--Bj\"orn--Bj\"orn--Shan\-mu\-ga\-lin\-gam~\cite{ABBSprime}
that if $\Om$ is finitely connected at the boundary, then
$\clOmm$ is homeomorphic to the prime end closure 
$\clOmP$ of $\Om$, using the definition of prime ends therein.
Since in this section, and the next, we only use the topology
on $\clOmm$ while the Newtonian spaces (and thus also the involved capacities) are 
only with respect to  $\Om=\Omm$ (and not $\overline{\Om}$), the results in these 
two sections can equivalently be formulated in terms of $\clOmP$ 
when $\Om$ is finitely connected at the boundary.
\end{remark}

\section{Resolutivity of functions on \texorpdfstring{$\partial_M\Om$}{}}
\label{sect-bdyM-Om}

\emph{In addition to the standing assumptions described at the end of
Section~\ref{sect-newcap}, we assume in this section,
as in Section~\ref{sect-Perron-Omm}, that\/ $\Om$ is a 
bounded domain which is finitely connected at the boundary
and that  $\Cp(X \setm \Om)>0$.}

\medskip

We now deduce some consequences of Theorem~\ref{thm-Newt-resolve-Omm}.

\begin{prop} \label{prop-Np-invariance-Omm}
Assume that $f: \clOmm \to \eR$ is $\bCp(\,\cdot\,;\Omm)$-quasi\-con\-tin\-u\-ous
and that $f|_{\Om} \in \Np(\Om)$. Assume further that 
$h: \bdym \Om \to \eR$ is zero  $\bCp(\,\cdot\,;\Omm)$-q.e.,
i.e.\ $\bCp(\{x \in \bdym \Om : h(x) \ne 0\};\Omm)=0$.
Then   $f+h$ is resolutive with respect to\/ $\Omm$ and
\[
   \Hpind{\Omm} (f+h) = \Hpind{\Omm} f.
\]
\end{prop}

\begin{proof}
Extend $h$ to $\Om$ by letting $h=0$ in $\Om$. 
Then $h$ is $\bCp(\,\cdot\,;\Omm)$-quasi\-con\-tin\-u\-ous, by (the $\Omm$ version of) 
Proposition~\ref{prop-outercap}.  The subadditivity of the $\bCp(\,\cdot\,;\Omm)$-capacity 
shows that also $f+h$ is $\bCp(\,\cdot\,;\Omm)$-quasi\-con\-tin\-u\-ous. Moreover 
$h \in \Np(\Om)$. 

Since $f+h=f$ in $\Om$ we have $H(f+h) = Hf$.
Theorem~\ref{thm-Newt-resolve-Omm} applied to both $f$ and $f+h$ 
shows that $f+h$ is resolutive with respect to $\Omm$ and that
\[
         \Hpind{\Omm} (f+h) = \oHp (f+h) = \oHp f = \Hpind{\Omm} f.  \qedhere
\]
\end{proof}

\begin{thm} \label{thm-cts-resolve-Omm}
Let $f \in C(\bdym \Om)$ and $h: \bdym \Om \to \eR$ be a function which is  zero 
$\bCp(\,\cdot\,;\Omm)$-q.e.\ on $\bdym \Om$. Then
$f$ and $f+h$ are resolutive with respect to\/ $\Omm$ and
\[
    \Hpind{\Omm} (f+h)  = \Hpind{\Omm} f.
\]
\end{thm}

\begin{proof}
For each $j=1,2,\ldots$, there is a Lipschitz function 
$f_j\in\Lip(\partial_M\Om)$ such that $f-1/j \le f_j\le f+1/j $ on $\bdym \Om$.
We can extend $f_j$ to be a Lipschitz function on $\clOmm$, so
$f_j\in\Lip(\clOmm)\subset\Np(\clOmm)$. 
It follows directly from Definition~\ref{def-Perron}
that $\uHpind{\Omm} f -1/j \le \uHpind{\Omm} f_j \le \uHpind{\Omm} f +1/j$, and hence
$\uHpind{\Omm} f_j \to \uHpind{\Omm} f$ uniformly, as $j \to \infty$.
The uniform convergences of $\lHpind{\Omm} f_j$, $\uHpind{\Omm} (f_j+h)$ and
$\lHpind{\Omm} (f_j +h)$ are proved in the same way. As $f_j \in \Np(\Om)$,
we have by Proposition~\ref{prop-Np-invariance-Omm}
that $\Hpind{\Omm} (f_j+h) =  \Hpind{\Omm} f_j$.
Letting $j \to \infty$ completes the proof.
\end{proof}

As a  consequence of these two results we obtain the  following uniqueness result.

\begin{cor} \label{cor-unique-Newt-Omm}
Let either $f: \clOmm \to \R$ be a bounded 
$\bCp(\,\cdot\,;\Omm)$-quasi\-con\-tin\-u\-ous function such
that $f|_{\Om} \in \Np(\Om)$, or $f \in C(\bdym \Om)$. 
Let $u$ be a bounded \p-harmonic function in\/ $\Om$. If  there is a set 
$E \subset \bdym \Om$ with $\bCp(E;\Omm)=0$ such that
\[
    \lim_{\Om \ni y \dMto x} u(y) = f(x) \quad \text{for all } x \in \bdym \Om \setm E,
\]
then $u=\Hpind{\Omm} f$.
\end{cor}

Note that if the word \emph{bounded} is omitted,
the result becomes false; consider for example, the Poisson 
kernel in the unit disc $B(0,1) \subset \C=\R^2$
with a pole at $1$ but vanishing on $\bdy B(0,1) \setm \{1\}$.

\begin{proof}
By adding a sufficiently large constant to both $f$ and $u$,
and then rescaling them simultaneously we may assume
without loss of generality that $0 \le u \le 1$ and
$0\le f\le 1$. Hence $u \in \UU_{f-\chi_E}(\Omm)$ and 
$u \in \LL_{f+\chi_E}(\Omm)$. Therefore, by 
Proposition~\ref{prop-Np-invariance-Omm} if $f$ is 
$\bCp(\,\cdot\,;\Omm)$-quasi\-con\-tin\-u\-ous and bounded with
$f|_\Om\in\Np(\Om)$, and Theorem~\ref{thm-cts-resolve-Omm} 
in the case when $f\in C(\bdym\Om)$, we see that 
\[
    u \ge \uHpind{\Omm} (f-\chi_E)=\Hpind{\Omm} f = \lHpind{\Omm} (f+\chi_E) \ge u.
\qedhere
\]
\end{proof}

The proofs of the following results are similar to the proof of 
Theorem~\ref{thm-cts-resolve-Omm}, and are left to the reader to verify.

\begin{prop} \label{prop-unif-Omm}
Let $f_j : \bdym \Om \to \eR$, $j=1,2,\ldots$\,, be  
resolutive functions with respect to\/ $\Omm$
and assume that $f_j \to f$ uniformly on $\bdym \Om$.
Then $f$ is resolutive with respect to\/ $\Omm$ and 
$\Hpind{\Omm} f_j \to \Hpind{\Omm} f$ uniformly in\/ $\Om$.
\end{prop}

\begin{prop} \label{prop-unif-pert-Omm} 
Let $f_j: \clOmm \to \eR$ be $\bCp(\,\cdot\,;\Omm)$-quasi\-con\-tin\-u\-ous functions
 such that $f_j|_{\Om} \in \Np(\Om)$, $j=1,2,\ldots$\,.
Assume also that $f_j \to f$ uniformly on $\bdym \Om$ as $j \to \infty$.
Let $h: \bdym \Om \to \eR$ be a function which is zero 
$\bCp(\,\cdot\,;\Omm)$-q.e.\ on $\bdym \Om$.
Then $f$ and $f+h$ are resolutive with respect to\/ $\Omm$ and
$\Hpind{\Omm} f = \Hpind{\Omm} (f+h)$.
\end{prop}

\section{Resolutivity of functions on \texorpdfstring{$\partial\Om$}{}}
\label{sect-resol-bdryOm}

\emph{In addition to the standing assumptions described at the end of
Section~\ref{sect-newcap}, we assume in this section that\/ $\Om$ is a 
nonempty bounded open set\/ {\rm (}not necessarily finitely connected 
at the boundary\/{\rm)} and that  $\Cp(X \setm \Om)>0$.}

\medskip

The results in Section~\ref{sect-bdyM-Om} have analogs for
Perron solutions with respect to the ordinary boundary $\bdy\Om$. 
Versions of these counterparts appear in 
Bj\"orn--Bj\"orn--Shan\-mu\-ga\-lin\-gam~\cite{BBS2} and 
Bj\"orn--Bj\"orn~\cite[Chapter~10]{BBbook}
under more restrictive assumptions such as $f\in \Np(X)$
and $f\in \Np(\clOm)$ respectively.
The capacities considered there are $\Cp$ and $\Cp(\,\cdot\,;\clOm)$.
The generalizations below have been made possible by 
the introduction of the new capacity $\bCp(\,\cdot\,;\Om)$ in this paper.
For a comparison of these results see the examples in 
Section~\ref{sect-examples}.

We next list these generalizations without proofs since the verification of these
results follow directly along the lines of 
the proofs of Theorem~\ref{thm-Newt-resolve-Omm}
and the results in Section~\ref{sect-bdyM-Om}.
For readers only interested in the results in this section,
we point out that some details are easier, for the only result
from Sections~\ref{sect-dM} and~\ref{sect-NpOmm}
needed for the results in this section is Proposition~\ref{prop-Np0-qcont-Omm},
which however has a much simpler proof in this case. 
Moreover, there is no need to assume that  $\Om$ is connected or finitely 
connected at the boundary for the results in this section.

\begin{thm} \label{thm-Np-invariance-Om}
Let $f: \clOm \to \eR$ be $\bCp(\,\cdot\,;\Om)$-quasi\-con\-tin\-u\-ous
with $f|_{\Om} \in \Np(\Om)$. Assume that $h: \bdy \Om \to \eR$ is zero 
$\bCp(\,\cdot\,;\Om)$-q.e. Then $f$ and  $f+h$ are resolutive
with respect to\/ $\Om$ and
\[
     \Hpind{\Om} (f+h) = \Hpind{\Om} f = \oHp f.
\]
\end{thm}

\begin{prop} \label{prop-cts-resolve-Om}
Let $f \in C(\bdy \Om)$ and let $h$ be a function which is zero 
$\bCp(\,\cdot\,;\Om)$-q.e.\ on $\bdy \Om$. Then 
$f+h$ is resolutive with respect to\/ $\Om$ and
\[
    \Hpind{\Om} (f+h)  = \Hpind{\Om} f.
\]
\end{prop}

Note that the resolutivity of $f \in C(\bdy \Om)$ was already obtained in 
Bj\"orn--Bj\"orn--Shan\-mu\-ga\-lin\-gam~\cite[Theorem~6.1]{BBS2}.

\begin{cor} \label{cor-unique-Om}
Let either $f: \clOm \to \R$ be a bounded 
$\bCp(\,\cdot\,;\Om)$-quasi\-con\-tin\-u\-ous function such
that $f|_{\Om} \in \Np(\Om)$, or $f \in C(\bdy \Om)$. Assume also 
that $u$ is a bounded \p-harmonic function in\/ $\Om$ and that there is a set 
$E \subset \bdy \Om$ with $\bCp(E;\Om)=0$ such that 
\[
    \lim_{\Om \ni y \to  x} u(y) = f(x) \quad \text{for all } x \in \bdy \Om \setm E.
\]
Then $u=\Hpind{\Om} f$.
\end{cor}

The following is a convenient existence and uniqueness result
for solutions of the Dirichlet problem with continuous boundary data.
With $\bCp(\,\cdot\,;\Om)$ replaced by $\Cp(\,\cdot\,)$
it was probably first given explicitly in Bj\"orn--Bj\"orn~\cite{BB2} as a 
consequence of Corollary~6.2  in Bj\"orn--Bj\"orn--Shan\-mu\-ga\-lin\-gam~\cite{BBS2} 
and the Kellogg property.

\begin{cor}  \label{cor-9.4}
Assume that $f \in C(\bdy \Om)$. Then there is a unique
bounded \p-harmonic function $u$ on\/ $\Om$ such that 
\begin{equation} \label{eq-unique-Om}
    \lim_{\Om \ni y \to  x} u(y) = f(x) \quad \text{for $\bCp(\,\cdot\,;\Om)$-q.e. } x \in \bdy \Om.
\end{equation}
Moreover, $u=\Hpind{\Om} f$.
\end{cor}

\begin{proof}
We already know by Corollary~\ref{cor-unique-Om} that if $u$
is a bounded \p-harmonic function on $\Om$ satisfying \eqref{eq-unique-Om}, 
then $u=\Hpind{\Om} f$, which shows the uniqueness.

As for the existence, let $u=\Hpind{\Om} f$.
That $f$ is resolutive with respect to $\Om$ follows
from Proposition~\ref{prop-cts-resolve-Om}, but was actually first shown 
in Bj\"orn--Bj\"orn--Shan\-mu\-ga\-lin\-gam~\cite[Theorem~6.1]{BBS2}. 
An application of Bj\"orn--Bj\"orn--Shan\-mu\-ga\-lin\-gam~\cite[Theorem~3.9]{BBS}
together with \cite[Theorem~6.1]{BBS2} shows 
that there is a set $E \subset \bdy \Om$ such that $\Cp(E)=0$ and
\[
    \lim_{\Om \ni y \to  x} u(y) = f(x) \quad \text{for all } x \in \bdy \Om \setm E.
\]
By Lemma~\ref{lem-Phicap-Omm}, $\bCp(E,\Om)=0$.
\end{proof}

\section{Examples and applications}
\label{sect-examples}

The results in this paper are the third generation of this type of results, following
Bj\"orn--Bj\"orn--Shan\-mu\-ga\-lin\-gam~\cite{BBS2}
and Bj\"orn--Bj\"orn~\cite{BBbook}.  In this section we give some new examples
illustrating the results of this paper (in some cases, in combination
with the results found in~\cite{BBbook}). These examples are not covered by the 
results found in~\cite{BBS2}. They also demonstrate
the differences between the capacities considered in this paper.
There are also some resolutivity results in
A.~Bj\"orn~\cite{ABjump} which are relevant for our discussion,
see Example~\ref{ex-cusp-new}.

\begin{example}
\label{ex-cusp-new}
(Cusps in $\R^2$) Let $X=\R^2$ (unweighted) and $p > 2$.
It is well-known that $\Cp(\{x\})>0$ for each $x \in \R^2$. Let $\Om$ be the cusp
\[
    \Om = \{(x_1,x_2): 0 < x_1 <1 \text{ and } 0 < x_2 < x_1^\be \}
\]
with $\beta >p-1$. By considering the functions 
$u_R(x_1,x_2)=\max\{1-x_1/R,0\}$, $0<R<1$, we see that
$\bCp(\{0\};\Om)=0$ and that $\Cp(\{0\};\overline{\Om})=0$.

This means that in Theorem~\ref{thm-Np-invariance-Om} and
Proposition~\ref{prop-cts-resolve-Om} 
we can change the boundary data arbitrarily at the origin
when considering the Perron solution with respect to  $\Om$,
even though $\Cp(\{0\})>0$. Similarly, the exceptional set $E$ in
Corollary~\ref{cor-unique-Om} can contain $0$.
This improves upon the perturbation results of
Bj\"orn--Bj\"orn--Shan\-mu\-ga\-lin\-gam~\cite{BBS2}. 

Note that any function in $\Np(\overline{\Om})$
is continuous apart from possibly at $0$. 
It thus follows from Propositions~\ref{prop-zero-cap-capacitable} 
and \ref{prop-outercap} that it is both $\Cp(\,\cdot\,;\overline{\Om})$- 
and $\bCp(\,\cdot\,;\Om)$-quasicontinuous, so that the resolutivity and perturbation 
results in Theorem~\ref{thm-Np-invariance-Om} apply to all functions in
$\Np(\overline{\Om})$. For functions in $C(\bdy \Om)$, these results follow also from  
A.~Bj\"orn~\cite[Theorem~1.3]{ABjump}.
However, the  function $f(x)=\sin \log |x| \in \Np(\overline{\Om})$
cannot be treated by~\cite{ABjump},  regardless of the value given to $f(0)$.
By Theorem~\ref{thm-Np-invariance-Om}, it is resolutive 
with respect to $\Om$ and $\Hp f$ is independent of $f(0)$.

Note that as $\Om$ is locally connected at the boundary, 
we have $\overline{\Om}=\clOmm$, and thus the $\Omm$-Perron solutions are
the same as the usual Perron solutions in this case. Consider instead
\[
\Om_0=(0,2)^2\setm \{(x_1,x_2): 0 < x_1 \le1 \text{ and } x_2 = x_1^\be \}
\supset \Om.
\]
Then $\clOmm_0\ne\overline{\Om}_0$ since the origin splits into $0_1$ and $0_2$ 
depending on whether it is approached from the right ($x_1$-direction) or from above 
($x_2$-direction). Then $\bCp(\{0\};\Om_0)$, $\Cp(\{0\};\overline{\Om}_0)$, 
$\bCp(\{0_2\};\Ommo)$ and  $\Cp(\{0_2\};\clOmm_0)$ are all positive and neither $0$ nor
$0_2$ can be treated by the results in Section~\ref{sect-resol-bdryOm}.
On the other hand, the discussion in the first paragraph of this example 
shows that $\bCp(\{0_1\};\Ommo)=\Cp(\{0_1\};\clOmm_0)=0$, so all the above 
resolutivity and perturbation results apply to $\{0_1\}$. For example, the function
\[
f_0(x)=\begin{cases}
	\sin \log |x|, & 
             \text{if $x\in\clOmm_0$, $|x|\le1$ and $d_M(x,0_1)\le d_M(x,0_2)$}, \\
	0,& \text{otherwise},
	\end{cases}
\]
is resolutive with respect to $\Ommo$ and $\Hp f$ is independent of $f(0_1)$.

To obtain similar results for $1<p\le2$, equip $\R^2$
with the measure $|x|^{-1} \, dx$ (which is doubling and supports  a \p-Poincar\'e inequality by
Heinonen--Kilpel\"ainen--Martio~\cite[Example~2.22]{HeKiMa})
and let $\be>p$ in the above construction.
\end{example}

\begin{example} \label{ex-comb-wide-thin}
(The topologist's comb) Let $\Om \subset \R^2$ be given by
\[
	\Om :=((0,2) \times (-1,1))
         \setm \bigl(\bigl\{1,\tfrac{1}{2},\tfrac{1}{4},
         \ldots,0\bigr\}
	\times [0,1)\bigr).
\]
Let $A=\{0\} \times (0,1] \subset \bdy \Om$.
Then $\overline{\Om}=[0,2] \times [-1,1]$ and so $\Cp(A;\overline{\Om})>0$. 
As we shall see,   $\bCp(A;\Om)=0$, and thus we get significantly better results using the
$\bCp(\,\cdot\,;\Om)$-capacity than with the $\Cp(\,\cdot\,;\overline{\Om})$-capacity.
To show that $\bCp(A;\Om)=0$ we let
\[
    \de_j=\biggl(\frac{3}{4}\biggr)^{j/(p-1)}
	\quad \text{and} \quad
	f_j(x,y)=\begin{cases}
	\displaystyle\min\biggl\{\frac{y}{\de_j},1\biggr\},
	 	& \text{if } 2^{-j}<x<2^{1-j},\ 0 <y<1, \\
	0, & \text{otherwise,} 
	\end{cases}
\]
for $j=1,2\ldots$\,. Set $h_k:=\sum_{j=k}^\infty f_j \in \A_A$.
Then $\|h_k\|_{L^p(\Om)}^p \to 0$ as $k \to \infty$, and
\[
	\|g_{h_k}\|_{L^p(\Om)}^p 
	= \sum_{j=k}^\infty \frac{2^{-j} \de_j}{\de_j^p} 
	= \sum_{j=k}^\infty \biggl(\frac{2}{3}\biggr)^j 
	\to 0,
	\quad \text{as } k \to \infty.
\]
Hence for all $p>1$, $\bCp(A;\Om) \le \|h_k\|_{\Np(\Om)}^p \to 0$, 
as $k \to \infty$. (Because singleton sets have zero capacity 
if and only if $1\le p\le 2$, it  also follows that
$\bCp(\itoverline{A},\Om)=0$ if and only if $1<p\le 2$.)

We can therefore perturb the boundary data as we wish on $A$ in 
Theorem~\ref{thm-Np-invariance-Om} and
Proposition~\ref{prop-cts-resolve-Om}, and  in
Corollary~\ref{cor-unique-Om}. In particular, 
Proposition~\ref{prop-cts-resolve-Om}  shows that if $f \in C(\bdy \Om)$ and
$h=f$ on $\bdy \Om \setm A$, then $h$ is resolutive 
with respect to $\Om$ and $\Hp h = \Hp f$. None of this can be inferred  
from the results of Bj\"orn--Bj\"orn--Shan\-mu\-ga\-lin\-gam~\cite{BBS2},
nor from the results in Bj\"orn--Bj\"orn~\cite{BBbook}. 

A variant of this example is obtained by replacing each slit
$S_j=\{2^{-j}\}\times[0,1)$ by the thin rectangle 
\[
R_j=[2^{-j}-2^{-j-2},2^{-j}+2^{-j-2}]\times[0,1), \quad j=0,1,\ldots,
\]
i.e.\ letting $\Om'=((0,2) \times (-1,1))\setm\bigcup_{j=0}^\infty R_j$. Since 
$g=\infty \chi_A$ is an upper gradient of $\chi_A$ in $\clOmprim$,
we see that $\Cp(A;\clOmprim) \le \|\chi_A\|_{\Np(\clOmprim)}=  0$.
Note that any curve starting in  $A$ and ending in $\clOmprim \setm A$ 
must pass through $(0,0)$ first, and therefore intersects $A$ along an interval of 
positive length. On the other hand, $\Cp(\itoverline{A},\clOmprim)=0$ 
if and only if $1<p\le 2$.

Thus, the above resolutivity and perturbation conclusions 
for this ``thickened'' comb $\Om'$ are  obtainable
already by the results in~\cite{BBbook}, where they are formulated
using $\Cp(\,\cdot\,;\clOmprim)$. However, we now show an interesting phenomenon 
on $\Om'$ which does not appear in the ordinary comb $\Om$:  For $(x,y)\in\clOm'$, let 
\begin{equation}  \label{eq-ex-comb-wide-thin}
     f(x,y) =\begin{cases}
            y, & \text{if } 0 \le y \le 1 \text{ and }
           x \in \bigcup_{j=1}^\infty (2^{-2j},2^{1-2j}), \\
            0, & \text{otherwise}.
       \end{cases}
\end{equation}
Then $g\equiv 1$ is an upper gradient of $f$
in $\clOmprim$ and hence $f \in \Np(\clOmprim)$.
Since $\bCp(A;\Om')=\Cp(A;\clOmprim) =0$ and $f$ is continuous at 
all points but those in $A$, Proposition~\ref{prop-outercap} implies that $f$ is both
$\bCp(\,\cdot\,;\Om')$- and $\Cp(\,\cdot\,;\clOmprim)$-quasi\-con\-tin\-u\-ous.
It is thus resolutive by Theorem~\ref{thm-Np-invariance-Om}
(and even by Theorem~10.15 in~\cite{BBbook}).

Note that $f$ oscillates near $A$, has countably many ``jumps'' on lines parallel to the 
$x$-axis with ordinate $0<y<1$,
and cannot be extended to a Newtonian function on $\R^2$.
A similar construction is not possible on the ordinary comb $\Om$ since
all the slits $\{2^{-j}\}\times[0,1)$  have positive $\bCp(\,\cdot\,;\Om)$-capacity
and thus a function with jumps at these slits cannot be 
$\bCp(\,\cdot\,;\Om)$-quasicontinuous on $\clOm$.
See however Example~\ref{ex-count-comb} below where a similar construction
is done on the countable comb. (The reason why it works there is that the union of the main slits
has zero $\bCp(\,\cdot\,,\Om)$-capacity, 
making $f$ $\bCp(\,\cdot\,;\Om)$-quasicontinuous on $\clOm$.)

The above disctinction between the ordinary comb and the ``thickened'' comb
further motivates Perron solutions with respect to the
Mazurkiewicz boundary $\bdym\Om$ and the generalized Perron solutions, 
see Sections~\ref{sect-Perron-Omm}, \ref{sect-bdyM-Om} 
and~\ref{sect-gen-Perron}, and also A.~Bj\"orn~\cite{ABcomb}.
\end{example}

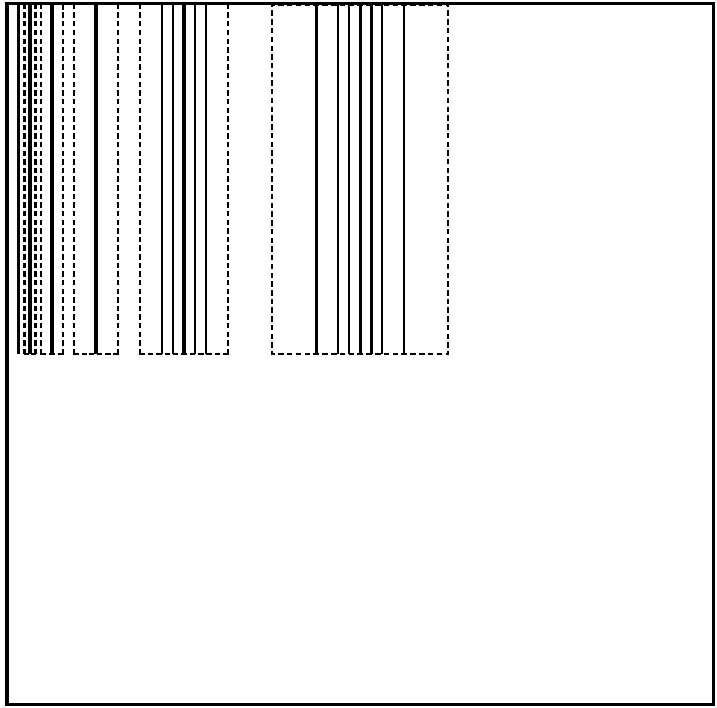
\begin{figure}[t]
\begin{center}
\setlength{\unitlength}{1800sp}
\begingroup\makeatletter\ifx\SetFigFont\undefined%
\gdef\SetFigFont#1#2#3#4#5{%
  \reset@font\fontsize{#1}{#2pt}%
  \fontfamily{#3}\fontseries{#4}\fontshape{#5}%
  \selectfont}%
\fi\endgroup%
\begin{picture}(9666,9687)(1168,-8794)
\thinlines
{\color[rgb]{0,0,0}\put(1201,-8761){\framebox(9600,9600){}}
}%
{\color[rgb]{0,0,0}\put(6001,839){\line( 0,-1){4800}}
}%
{\color[rgb]{0,0,0}\put(3601,839){\line( 0,-1){4800}}
}%
{\color[rgb]{0,0,0}\put(2401,839){\line( 0,-1){4800}}
}%
{\color[rgb]{0,0,0}\put(1801,839){\line( 0,-1){4800}}
}%
{\color[rgb]{0,0,0}\put(1501,839){\line( 0,-1){4800}}
}%
{\color[rgb]{0,0,0}\multiput(3001,839)(0.00000,-118.51852){41}{\line( 0,-1){ 59.259}}
\multiput(3001,-3961)(114.28571,0.00000){11}{\line( 1, 0){ 57.143}}
\multiput(4201,-3961)(0.00000,118.51852){41}{\line( 0, 1){ 59.259}}
}%
{\color[rgb]{0,0,0}\multiput(2101,839)(0.00000,-118.51852){41}{\line( 0,-1){ 59.259}}
\multiput(2101,-3961)(109.09091,0.00000){6}{\line( 1, 0){ 54.545}}
\multiput(2701,-3961)(0.00000,118.51852){41}{\line( 0, 1){ 59.259}}
}%
{\color[rgb]{0,0,0}\multiput(1651,839)(0.00000,-118.51852){41}{\line( 0,-1){ 59.259}}
\multiput(1651,-3961)(120.00000,0.00000){3}{\line( 1, 0){ 60.000}}
\multiput(1951,-3961)(0.00000,118.51852){41}{\line( 0, 1){ 59.259}}
\put(1951,839){\line( 0,-1){ 75}}
}%
{\color[rgb]{0,0,0}\multiput(1426,839)(0.00000,-118.51852){41}{\line( 0,-1){ 59.259}}
\multiput(1426,-3961)(100.00000,0.00000){2}{\line( 1, 0){ 50.000}}
\multiput(1576,-3961)(0.00000,118.51852){41}{\line( 0, 1){ 59.259}}
}%
\thicklines
{\color[rgb]{0,0,0}\put(6001,839){\line( 0,-1){4800}}
}%
{\color[rgb]{0,0,0}\put(3601,839){\line( 0,-1){4800}}
}%
{\color[rgb]{0,0,0}\put(2401,839){\line( 0,-1){4800}}
}%
{\color[rgb]{0,0,0}\put(1801,839){\line( 0,-1){4800}}
}%
{\color[rgb]{0,0,0}\put(1501,839){\line( 0,-1){4800}}
}%
{\color[rgb]{0,0,0}\put(1201,-8761){\framebox(9600,9600){}}
}%
\thinlines
{\color[rgb]{0,0,0}\put(6601,839){\line( 0,-1){4800}}
}%
{\color[rgb]{0,0,0}\put(6301,839){\line( 0,-1){4800}}
}%
{\color[rgb]{0,0,0}\put(6151,839){\line( 0,-1){4800}}
}%
{\color[rgb]{0,0,0}\put(5401,839){\line( 0,-1){4800}}
}%
{\color[rgb]{0,0,0}\put(5701,839){\line( 0,-1){4800}}
}%
{\color[rgb]{0,0,0}\put(5851,839){\line( 0,-1){4800}}
}%
\thicklines
{\color[rgb]{0,0,0}\put(1351,839){\line( 0,-1){4800}}
}%
\thinlines
{\color[rgb]{0,0,0}\put(4801,-3961){\dashbox{60}(2400,4800){}}
}%
{\color[rgb]{0,0,0}\put(3301,839){\line( 0,-1){4800}}
}%
{\color[rgb]{0,0,0}\put(3451,839){\line( 0,-1){4800}}
}%
{\color[rgb]{0,0,0}\put(3901,839){\line( 0,-1){4800}}
}%
{\color[rgb]{0,0,0}\put(3751,839){\line( 0,-1){4800}}
}%
\end{picture}

    \caption{\label{fig1} The topologist's combs in 
Examples~\ref{ex-comb-wide-thin} and~\ref{ex-count-comb}.
The thick lines represent the main slits $S_j$, the broken lines represent 
the rectangles $R_j$ in the ``thickened'' comb and the thin lines
represent the secondary slits near the first two main slits
$S_0$ and $S_1$ in the countable comb.}
  \end{center}
\end{figure}

\begin{example} 
\label{ex-double-cone}
(Double comb) Let $\Om \subset \R^2$ be given by
\[
     \Om = (-1,1)^2 \setm \bigl(\bigl\{\pm \tfrac{1}{2}, 
       \pm \tfrac{1}{4},  \pm \tfrac{1}{8},  
       \ldots, 0\bigr\} \times [0,1)\bigr).
\]
Just as in Example~\ref{ex-comb-wide-thin}  we find that
\[
      \bCp(A;\Om) = 0 < \Cp(A;\overline{\Om}),
\]
where $A=\{0\} \times (0,1]$, and
we get similar consequences for perturbing the boundary data at points in  $A$ as in 
Example~\ref{ex-comb-wide-thin} (not obtainable from 
the results in \cite{BBS2} and \cite{BBbook}). Consider e.g.\  the function 
\begin{equation} \label{eq-ex-double-cone}
   f(x,y)= \begin{cases}
        y, & 0  \le x \le  1, \ 0 < y \le 1, \\
        0, & \text{otherwise}.
     \end{cases}
\end{equation}
Arguing  as in Example~\ref{ex-comb-wide-thin}, we see that $f \in \Np(\Om)$ 
and that $f$ is $\bCp(\,\cdot\,,\Om)$-quasi\-con\-tin\-u\-ous.
Hence, by Theorem~\ref{thm-Np-invariance-Om}, $f$ is resolutive
even though $f \notin \Np(\overline{\Om})$ (since it is not absolutely
continuous on lines parallel to the $x$-axis with ordinate $0<y<1$).
Therefore, the resolutivity of $f$ \emph{cannot} be obtained by
the results in \cite{BBS2} or \cite{BBbook}.
\end{example}

\begin{example}   \label{ex-count-comb}
(Countable comb)
Let $\Om \subset \R^2$ be given by
\[
     \Om = ((0,2) \times (-1,1)) \setm 
          \biggl(\biggl(E \cup \bigcup_{j=0}^\infty E_j\biggr) \times [0,1)\biggr),
\]
where
\[
   E=\{2^{-j}: j=0,1,\ldots\}
   \quad \text{and} \quad
   E_{j}=\{2^{-j}(1 \pm 2^{-k}): k=3,4,\ldots\}.
\]
Furthermore, let $A'=(E \cup \{0\}) \times(0,1]$
and let $f$ be given by \eqref{eq-ex-comb-wide-thin}.
Note that $f$ oscillates near $\{0\}\times(0,1]$ and has countably many ``jumps'' on 
lines parallel to the $x$-axis with ordinate $0<y<1$.

As in Example~\ref{ex-comb-wide-thin} 
(and using also the countable subadditivity
of the capacity) we see that $\bCp(A',\Om)=0$, from which it follows
that $f$ is $\bCp(\,\cdot\,,\Om)$-quasi\-con\-tin\-u\-ous. Moreover, $f \in \Np(\Om)$.
Hence, $f$ is resolutive by Theorem~\ref{thm-Np-invariance-Om}.
\end{example}

We next give an example of a domain whose boundary $\bdy \Om \subset \R$ 
has positive measure, and such that there is a set $K \subset \bdy \Om$ with 
$\bCp(K;\Om)=0$ and $\mu(\bdy \Om \setm K)=0$, i.e.\ 
from a measure-theoretic point of view $K$ is essentially all the boundary, but
for the Perron solution results in this paper it is negligible. 

\begin{example} \label{ex-Cantor}
Let $c_n=2^{-n}$, $n=0,1,\ldots$. We shall construct a Cantor set 
$\Kt=\bigcap_{n=0}^\infty\Kt_n \subset[0,2]=\Kt_0$ inductively as follows.
The $n$th generation $\Kt_n$ consists of $2^n$ closed intervals of length 
$\al_n:=2^{-n}(1+c_n)=2^{-2n}(2^n+1)$, $n=0,1,\ldots$, and is obtained by removing
the open middle subinterval of length 
\[
\theta_n=\al_{n-1}-2\al_n = 2^{1-n}(c_{n-1}-c_n) = 2^{1-n} c_n = 2^{1-2n}
\] 
from each interval constituting $\Kt_{n-1}$. It is easy to see that $\Kt$ has positive length, 
or more precisely $\La_1(\Kt)=1$, where $\La_1$ is the 1-dimensional Lebesgue measure.

Set $K_n=\Kt_n\times\Kt_n$ and $K=\Kt\times\Kt$.
The set $K_n$ consists of $4^n$ closed squares $Q_{n,j}$, $j=1,\ldots,4^n$, and
$K$ has two-dimensional Lebesgue measure $\La_2(K)=1$.  Let  
\[
	x_{n,j}=\inf\{x: (x,y) \in Q_{n,j}\}
	\quad \text{and}\quad
	x'_{n,j}=\sup\{x: (x,y) \in Q_{n,j}\}=x_{n,j}+\alp_n
\]
be the left- and right-hand end points of  the projection of the square $Q_{n,j}$ to the $x$-axis, and set
\[
        A_{n,j,k}=\begin{cases}
	    \{(x,y) : x \ge x_{n,j}\}, & \text{if $k$ is even}, \\
	    \{(x,y) : x \le x'_{n,j}\}, & \text{if $k$ is odd}.
	\end{cases}
\]
Consider the sets 
\[
       F_{n,j,k} =
	\biggl\{z \in A_{n,j,k}: 
	 \dist(z,Q_{n,j})
	= \frac{\theta_n}{6} (1 + 2^{-n}k)\biggr\}, \quad k=0,\ldots,2^n,
\]
obtained by the concatenation of three line segments and two quartercircles.
Note that all the sets $F_{n,j,k}$ are pairwise disjoint
and in fact, $\dist(F_{n,j,k},F_{n,j,k+1})=2^{-n}\theta_n/6$. For $m>n$,
\[
  \dist(F_{n,j,k},F_{m,i,l})\ge \frac{\theta_n}{6}- \frac{\theta_{n+1}}{3}
= \frac{2^{-n}(c_n-c_{n+1})}{3} = \frac{\theta_{n+1}}{3} \ge \frac{2^{-n}\theta_n}{6}.
\]
Finally we let 
\[
	\Om=(-1,3)^2 \setm \biggl(K \cup \bigcup_{n=1}^\infty 
	    \bigcup_{j=1}^{4^n} \bigcup_{k=1}^{2^n+1} F_{n,j,k} \biggr).
\]
Around each square $Q_{n,j}$ we have thus placed $2^{n}+1$ pairwise disjoint arcs with
openings placed alternately to the left and to the right of the square.
This means that any curve in $\Om$ connecting a point in the region outside of all these 
arcs with  a point in $Q_{n,j}$ must have length at least $2^n \alp_n = 1+c_n$.
Thus, any curve starting at a point $z\in\Om$
with $\dist(z,Q_{n,j}) \ge\theta_n/3$ and ending at a point $w\in\Om$ with 
$\dist(w,Q_{m,i}) \le \theta_{m}/6$, $m>n$, must have length at least 
$\sum_{k=n}^m 2^k\alp_k\to\infty$ as $m\to\infty$, for each $n$.  

We will now show that $\bCp(K;\Om)=0$. Let
\[
    f_j(z)=\min\{\din(z,(-1,0))/j,1\}.
\]
By the observation above, we see that $f_j \in \A_{K}$ with  $g_{f_j} \le1/j$.
Moreover, $\|f_j\|_{L^p(\Om)}^p \to 0$ and 
$\|g_{f_j}\|_{L^p(\Om)}^p \le \mu(\Om)/j^p \to 0$ as $j\to \infty$.
Hence $\bCp(K;\Om) \le \|f_j\|_{\Np(\Om)}^p \to 0$ as $j\to \infty$. On the other hand, 
\[
\Cp(K;\overline{\Om}) = \Cp(K;[-1,3]^2) \ge \mu(K)>0.
\]
In fact, we even have $\Cp(K;(\overline{\Om};\mu_0))>0$. Indeed,
if $\Cp(\Kt;[-1,3] \times [-1,0])$ were $0$, then a reflection
and localization argument would show that $\Cp(\Kt;\R^2)$ would be $0$,
which contradicts Theorem~2.26 in  Heinonen--Kilpel\"ainen--Martio~\cite{HeKiMa} since $\La_1(\Kt)>0$.
Hence $0 < \Cp(\Kt;[-1,3] \times [-1,0])  \le \Cp(K;(\overline{\Om};\mu_0))$, by monotonicity.

Observe that $\Om$ is boundedly connected at the boundary, and that
if $x \in \bdy \Om$, then $\Om$ is either
locally connected at $x$ (for $x\in\bdry(-1,3)^2\cup K$) or $2$-connected at $x$ (for $x\in F_{n,j,k})$.
\end{example}

The following modification of the example above may be of interest.

\begin{example} \label{ex-Cantor-2}
Let $\Om$ be constructed just as in Example~\ref{ex-Cantor},
but replace each $F_{n,j,k}$ by $F_{n,j,k}'=\itoverline{G}_{n,j,k}$, where
${G}_{n,j,k}$ is the $2^{-n}\theta_n/24$-neighbourhood of $F_{n,j,k}$.
Then all $F'_{n,j,k}$ are still pairwise disjoint and $\Om$ is locally connected
at the boundary.

This time, not only $\bCp(K;\Om)=0$ but also $\Cp(K;(\overline{\Om};\mu_0))=0$.
Indeed, there are no nonconstant rectifiable curves 
in $\overline{\Om}=[-1,3]^2\setm\bigcup_{n,j,k}G_{n,j,k}$ 
intersecting $K$, and hence $\chi_{K} \in \Np(\overline{\Om})$. Note however that  
$0<\mu(K) \le \Cp(K;\overline{\Om}) \le  \|\chi_{K}\|_{\Np(\overline{\Om})}^p = \mu(K)$.
\end{example}

In Examples~\ref{ex-Cantor} and~\ref{ex-Cantor-2} we had $\bCp(K;\Om)=0$
because there were no rectifiable curves in $\Om$ terminating in $K$.
In the following example exery $x\in K$ is accessible by rectifiable curves 
from $\Om$ but there is still a set $K^*\subset K$ with full measure in $K$
such that $\bCp(K^*;\Om)=0$. Roughly speaking, $K^*$ lies deep in $K$ and there are 
rather few curves reaching that far.

\begin{figure}[t]
\begin{center}
\setlength{\unitlength}{1800sp}%
\begingroup\makeatletter\ifx\SetFigFont\undefined%
\gdef\SetFigFont#1#2#3#4#5{%
  \reset@font\fontsize{#1}{#2pt}%
  \fontfamily{#3}\fontseries{#4}\fontshape{#5}%
  \selectfont}%
\fi\endgroup%
\begin{picture}(9624,9624)(1189,-9973)
\thinlines
{\color[rgb]{0,0,0}\put(1201,-961){\framebox(600,600){}}
}%
{\color[rgb]{0,0,0}\put(1876,-961){\framebox(600,600){}}
}%
{\color[rgb]{0,0,0}\put(3526,-961){\framebox(600,600){}}
}%
{\color[rgb]{0,0,0}\put(4201,-961){\framebox(600,600){}}
}%
{\color[rgb]{0,0,0}\put(1876,-1636){\framebox(600,600){}}
}%
{\color[rgb]{0,0,0}\put(3526,-1636){\framebox(600,600){}}
}%
{\color[rgb]{0,0,0}\put(1876,-3286){\framebox(600,600){}}
}%
{\color[rgb]{0,0,0}\put(1876,-3961){\framebox(600,600){}}
}%
{\color[rgb]{0,0,0}\put(1201,-3286){\framebox(600,600){}}
}%
{\color[rgb]{0,0,0}\put(3526,-3286){\framebox(600,600){}}
}%
{\color[rgb]{0,0,0}\put(4201,-3286){\framebox(600,600){}}
}%
{\color[rgb]{0,0,0}\put(3526,-3961){\framebox(600,600){}}
}%
{\color[rgb]{0,0,0}\put(7876,-961){\framebox(600,600){}}
}%
{\color[rgb]{0,0,0}\put(7201,-961){\framebox(600,600){}}
}%
{\color[rgb]{0,0,0}\put(9526,-1636){\framebox(600,600){}}
}%
{\color[rgb]{0,0,0}\put(10201,-1636){\framebox(600,600){}}
}%
{\color[rgb]{0,0,0}\put(9526,-961){\framebox(600,600){}}
}%
{\color[rgb]{0,0,0}\put(7876,-3286){\framebox(600,600){}}
}%
{\color[rgb]{0,0,0}\put(7876,-3961){\framebox(600,600){}}
}%
{\color[rgb]{0,0,0}\put(7201,-3286){\framebox(600,600){}}
}%
{\color[rgb]{0,0,0}\put(10201,-3286){\framebox(600,600){}}
}%
{\color[rgb]{0,0,0}\put(9526,-3961){\framebox(600,600){}}
}%
{\color[rgb]{0,0,0}\put(1201,-6961){\framebox(600,600){}}
}%
{\color[rgb]{0,0,0}\put(1876,-6961){\framebox(600,600){}}
}%
{\color[rgb]{0,0,0}\put(1876,-7636){\framebox(600,600){}}
}%
{\color[rgb]{0,0,0}\put(3526,-6961){\framebox(600,600){}}
}%
{\color[rgb]{0,0,0}\put(4201,-6961){\framebox(600,600){}}
}%
{\color[rgb]{0,0,0}\put(4201,-7636){\framebox(600,600){}}
}%
{\color[rgb]{0,0,0}\put(7876,-6961){\framebox(600,600){}}
}%
{\color[rgb]{0,0,0}\put(7876,-7636){\framebox(600,600){}}
}%
{\color[rgb]{0,0,0}\put(7201,-7636){\framebox(600,600){}}
}%
{\color[rgb]{0,0,0}\put(9526,-6961){\framebox(600,600){}}
}%
{\color[rgb]{0,0,0}\put(9526,-7636){\framebox(600,600){}}
}%
{\color[rgb]{0,0,0}\put(10201,-7636){\framebox(600,600){}}
}%
{\color[rgb]{0,0,0}\put(1876,-9286){\framebox(600,600){}}
}%
{\color[rgb]{0,0,0}\put(1876,-9961){\framebox(600,600){}}
}%
{\color[rgb]{0,0,0}\put(1201,-9961){\framebox(600,600){}}
}%
{\color[rgb]{0,0,0}\put(3526,-9286){\framebox(600,600){}}
}%
{\color[rgb]{0,0,0}\put(4201,-9286){\framebox(600,600){}}
}%
{\color[rgb]{0,0,0}\put(3526,-9961){\framebox(600,600){}}
}%
{\color[rgb]{0,0,0}\put(7201,-9286){\framebox(600,600){}}
}%
{\color[rgb]{0,0,0}\put(7876,-9286){\framebox(600,600){}}
}%
{\color[rgb]{0,0,0}\put(7876,-9961){\framebox(600,600){}}
}%
{\color[rgb]{0,0,0}\put(9526,-9286){\framebox(600,600){}}
}%
{\color[rgb]{0,0,0}\put(10201,-9286){\framebox(600,600){}}
}%
{\color[rgb]{0,0,0}\put(9526,-9961){\framebox(600,600){}}
}%
{\color[rgb]{0,0,0}\put(1201,-1636){\framebox(600,600){}}
}%
{\color[rgb]{0,0,0}\put(4201,-1636){\framebox(600,600){}}
}%
{\color[rgb]{0,0,0}\put(7201,-1636){\framebox(600,600){}}
}%
{\color[rgb]{0,0,0}\put(10201,-961){\framebox(600,600){}}
}%
{\color[rgb]{0,0,0}\put(10201,-3961){\framebox(600,600){}}
}%
{\color[rgb]{0,0,0}\put(7201,-3961){\framebox(600,600){}}
}%
{\color[rgb]{0,0,0}\put(4201,-3961){\framebox(600,600){}}
}%
{\color[rgb]{0,0,0}\put(1201,-3961){\framebox(600,600){}}
}%
{\color[rgb]{0,0,0}\put(1201,-7636){\framebox(600,600){}}
}%
{\color[rgb]{0,0,0}\put(3526,-7636){\framebox(600,600){}}
}%
{\color[rgb]{0,0,0}\put(7201,-6961){\framebox(600,600){}}
}%
{\color[rgb]{0,0,0}\put(10201,-6961){\framebox(600,600){}}
}%
{\color[rgb]{0,0,0}\put(10201,-9961){\framebox(600,600){}}
}%
{\color[rgb]{0,0,0}\put(7201,-9961){\framebox(600,600){}}
}%
{\color[rgb]{0,0,0}\put(4201,-9961){\framebox(600,600){}}
}%
{\color[rgb]{0,0,0}\put(1201,-9286){\framebox(600,600){}}
}%
{\color[rgb]{0,0,0}\put(3526,-7636){\dashbox{60}(4950,4950){}}
}%
{\color[rgb]{0,0,0}\put(1876,-3286){\dashbox{60}(2250,2250){}}
}%
{\color[rgb]{0,0,0}\put(7876,-3286){\dashbox{60}(2250,2250){}}
}%
{\color[rgb]{0,0,0}\put(1876,-9286){\dashbox{60}(2250,2250){}}
}%
{\color[rgb]{0,0,0}\put(7876,-9286){\dashbox{60}(2250,2250){}}
}%
{\color[rgb]{0,0,0}\put(1501,-1336){\dashbox{60}(675,675){}}
}%
{\color[rgb]{0,0,0}\put(3826,-1336){\dashbox{60}(675,675){}}
}%
{\color[rgb]{0,0,0}\put(1501,-3661){\dashbox{60}(675,675){}}
}%
{\color[rgb]{0,0,0}\put(9826,-1336){\dashbox{60}(675,675){}}
}%
{\color[rgb]{0,0,0}\put(1501,-7336){\dashbox{60}(675,675){}}
}%
{\color[rgb]{0,0,0}\put(1501,-9661){\dashbox{60}(675,675){}}
}%
{\color[rgb]{0,0,0}\put(3826,-9661){\dashbox{60}(675,675){}}
}%
{\color[rgb]{0,0,0}\put(7501,-9661){\dashbox{60}(675,675){}}
}%
{\color[rgb]{0,0,0}\put(9826,-7336){\dashbox{60}(675,675){}}
}%
{\color[rgb]{0,0,0}\put(1351,-811){\dashbox{60}(300,300){}}
}%
{\color[rgb]{0,0,0}\put(2026,-811){\dashbox{60}(300,300){}}
}%
{\color[rgb]{0,0,0}\put(1351,-1486){\dashbox{60}(300,300){}}
}%
{\color[rgb]{0,0,0}\put(3676,-811){\dashbox{60}(300,300){}}
}%
{\color[rgb]{0,0,0}\put(4351,-811){\dashbox{60}(300,300){}}
}%
{\color[rgb]{0,0,0}\put(4351,-1486){\dashbox{60}(300,300){}}
}%
{\color[rgb]{0,0,0}\put(7351,-811){\dashbox{60}(300,300){}}
}%
{\color[rgb]{0,0,0}\put(8026,-811){\dashbox{60}(300,300){}}
}%
{\color[rgb]{0,0,0}\put(9676,-811){\dashbox{60}(300,300){}}
}%
{\color[rgb]{0,0,0}\put(10351,-811){\dashbox{60}(300,300){}}
}%
{\color[rgb]{0,0,0}\put(10351,-1486){\dashbox{60}(300,300){}}
}%
{\color[rgb]{0,0,0}\put(10351,-3136){\dashbox{60}(300,300){}}
}%
{\color[rgb]{0,0,0}\put(10351,-3811){\dashbox{60}(300,300){}}
}%
{\color[rgb]{0,0,0}\put(9676,-6811){\dashbox{60}(300,300){}}
}%
{\color[rgb]{0,0,0}\put(10351,-6811){\dashbox{60}(300,300){}}
}%
{\color[rgb]{0,0,0}\put(10351,-7486){\dashbox{60}(300,300){}}
}%
{\color[rgb]{0,0,0}\put(1351,-6811){\dashbox{60}(300,300){}}
}%
{\color[rgb]{0,0,0}\put(2026,-6811){\dashbox{60}(300,300){}}
}%
{\color[rgb]{0,0,0}\put(1351,-7486){\dashbox{60}(300,300){}}
}%
{\color[rgb]{0,0,0}\put(10351,-9136){\dashbox{60}(300,300){}}
}%
{\color[rgb]{0,0,0}\put(10351,-9811){\dashbox{60}(300,300){}}
}%
{\color[rgb]{0,0,0}\put(8026,-9811){\dashbox{60}(300,300){}}
}%
{\color[rgb]{0,0,0}\put(7351,-9811){\dashbox{60}(300,300){}}
}%
{\color[rgb]{0,0,0}\put(7351,-9136){\dashbox{60}(300,300){}}
}%
{\color[rgb]{0,0,0}\put(1351,-9136){\dashbox{60}(300,300){}}
}%
{\color[rgb]{0,0,0}\put(1351,-9811){\dashbox{60}(300,300){}}
}%
{\color[rgb]{0,0,0}\put(2026,-9811){\dashbox{60}(300,300){}}
}%
{\color[rgb]{0,0,0}\put(3676,-9811){\dashbox{60}(300,300){}}
}%
{\color[rgb]{0,0,0}\put(4351,-9811){\dashbox{60}(300,300){}}
}%
{\color[rgb]{0,0,0}\put(4351,-9136){\dashbox{60}(300,300){}}
}%
{\color[rgb]{0,0,0}\put(9826,-9661){\dashbox{60}(675,675){}}
}%
{\color[rgb]{0,0,0}\put(9676,-9811){\dashbox{60}(300,300){}}
}%
{\color[rgb]{0,0,0}\put(1351,-3136){\dashbox{60}(300,300){}}
}%
{\color[rgb]{0,0,0}\put(1351,-3811){\dashbox{60}(300,300){}}
}%
{\color[rgb]{0,0,0}\put(2026,-3811){\dashbox{60}(300,300){}}
}%
{\color[rgb]{0,0,0}\put(7501,-1336){\dashbox{60}(675,675){}}
}%
{\color[rgb]{0,0,0}\put(7876,-1636){\framebox(600,600){}}
}%
{\color[rgb]{0,0,0}\put(9826,-3661){\dashbox{60}(675,675){}}
}%
{\color[rgb]{0,0,0}\put(9526,-3286){\framebox(600,600){}}
}%
{\color[rgb]{0,0,0}\put(7351,-1486){\dashbox{60}(300,300){}}
}%
{\color[rgb]{0,0,0}\put(9676,-3811){\dashbox{60}(300,300){}}
}%
{\color[rgb]{0,0,0}\put(2026,-1486){\dashbox{60}(300,300){}}
}%
{\color[rgb]{0,0,0}\put(3676,-1486){\dashbox{60}(300,300){}}
}%
{\color[rgb]{0,0,0}\put(2026,-3136){\dashbox{60}(300,300){}}
}%
{\color[rgb]{0,0,0}\put(3676,-3136){\dashbox{60}(300,300){}}
}%
{\color[rgb]{0,0,0}\put(4351,-3136){\dashbox{60}(300,300){}}
}%
{\color[rgb]{0,0,0}\put(3826,-3661){\dashbox{60}(675,675){}}
}%
{\color[rgb]{0,0,0}\put(3676,-3811){\dashbox{60}(300,300){}}
}%
{\color[rgb]{0,0,0}\put(4351,-3811){\dashbox{60}(300,300){}}
}%
{\color[rgb]{0,0,0}\put(8026,-1486){\dashbox{60}(300,300){}}
}%
{\color[rgb]{0,0,0}\put(9676,-1486){\dashbox{60}(300,300){}}
}%
{\color[rgb]{0,0,0}\put(7351,-3136){\dashbox{60}(300,300){}}
}%
{\color[rgb]{0,0,0}\put(8026,-3136){\dashbox{60}(300,300){}}
}%
{\color[rgb]{0,0,0}\put(7351,-3811){\dashbox{60}(300,300){}}
}%
{\color[rgb]{0,0,0}\put(8026,-3811){\dashbox{60}(300,300){}}
}%
{\color[rgb]{0,0,0}\put(7501,-3661){\dashbox{60}(675,675){}}
}%
{\color[rgb]{0,0,0}\put(9676,-3136){\dashbox{60}(300,300){}}
}%
{\color[rgb]{0,0,0}\put(2026,-7486){\dashbox{60}(300,300){}}
}%
{\color[rgb]{0,0,0}\put(3676,-6811){\dashbox{60}(300,300){}}
}%
{\color[rgb]{0,0,0}\put(4351,-6811){\dashbox{60}(300,300){}}
}%
{\color[rgb]{0,0,0}\put(3676,-7486){\dashbox{60}(300,300){}}
}%
{\color[rgb]{0,0,0}\put(4351,-7486){\dashbox{60}(300,300){}}
}%
{\color[rgb]{0,0,0}\put(3826,-7336){\dashbox{60}(675,675){}}
}%
{\color[rgb]{0,0,0}\put(7351,-6811){\dashbox{60}(300,300){}}
}%
{\color[rgb]{0,0,0}\put(8026,-6811){\dashbox{60}(300,300){}}
}%
{\color[rgb]{0,0,0}\put(9676,-7486){\dashbox{60}(300,300){}}
}%
{\color[rgb]{0,0,0}\put(7351,-7486){\dashbox{60}(300,300){}}
      }%
      {\color[rgb]{0,0,0}\put(8026,-7486){\dashbox{60}(300,300){}}
      }%
      {\color[rgb]{0,0,0}\put(7501,-7336){\dashbox{60}(675,675){}}
      }%
      {\color[rgb]{0,0,0}\put(9676,-9136){\dashbox{60}(300,300){}}
      }%
      {\color[rgb]{0,0,0}\put(8026,-9136){\dashbox{60}(300,300){}}
      }%
      {\color[rgb]{0,0,0}\put(3676,-9136){\dashbox{60}(300,300){}}
      }%
      {\color[rgb]{0,0,0}\put(2026,-9136){\dashbox{60}(300,300){}}
      }%
    \end{picture}%
    \caption{\label{fig2} The squares $Q^*$ of the first four generations 
in the construction in Example~\ref{ex-Jana} are
drawn by broken lines, and solid lines mark the set $K_3$.}
  \end{center}
\end{figure}
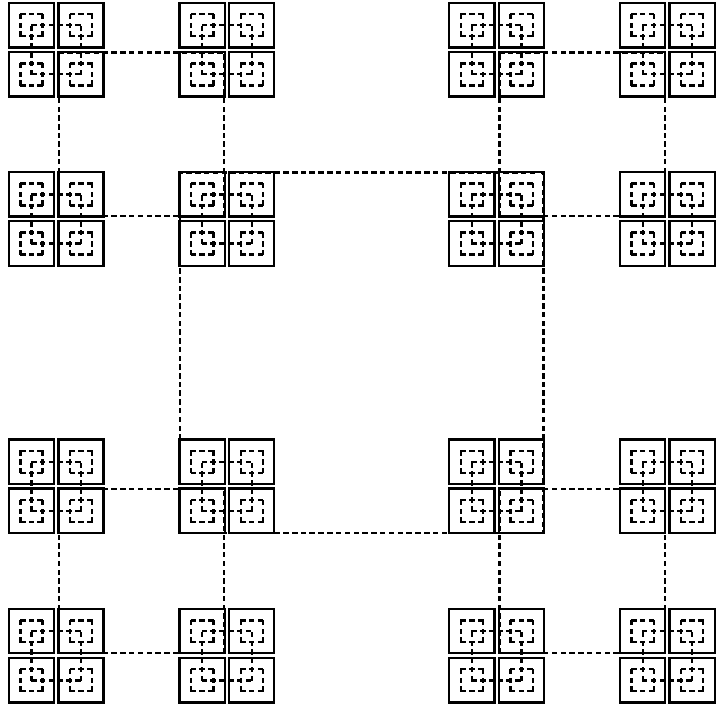

\begin{example} \label{ex-Jana}
Let $\{c_n\}_{n=0}^\infty$ be a strictly decreasing sequence such that
$0<c_0\le\frac{1}{3}$ and $\lim_{n\to\infty}c_n=0$. We shall construct a Cantor 
set $\Kt=\bigcap_{n=0}^\infty\Kt_n \subset[0,1+c_0]=\Kt_0$ inductively as follows.
The $n$th generation $\Kt_n$ consists of $2^n$ closed intervals of length 
$\al_n=2^{-n}(1+c_n)$, $n=0,1,\ldots$, and is obtained by removing
the open middle subinterval of length $\al_{n-1}-2\al_n$ from each
interval constituting $\Kt_{n-1}$.
It is easy to see that $\Kt$ has length 1, and in fact
\begin{equation}  \label{eq-length-Kt}
   \La_1([0,\al_n]\cap\Kt) = \lim_{k\to\infty} 2^{k-n} \al_k
       = \lim_{k\to\infty} 2^{k-n} 2^{-k} (1+c_k) = 2^{-n}, \quad n=0,1,\ldots,
\end{equation}
where $\La_1$ is the 1-dimensional Lebesgue measure.

Set $K_n=\Kt_n\times\Kt_n$ and $K=\Kt\times\Kt$.
The set $K_n$ consists of $4^n$ closed squares $Q_{n,j}$, $j=1,\ldots,4^n$, and
$K$ has area 1, by \eqref{eq-length-Kt}. Let 
\[
    \Om= (-1,3)^2\setm K \subset\R^2.
\]
We shall now construct a set $K^*\subset K$ with area 1 and
$\bCp(K^*;\Om)=0$ simultaneously for all $p>1$.
Fix $n\ge0$ and let $Q=Q_{n,j}$ be one of the $4^{n}$ closed squares
of sidelength $\al_{n}$, constituting $K_n$.
Let $Q^*=Q_{n,j}^*$ be the square concentric with $Q$ and of sidelength
$\be_n = \al_{n}-2\al_{n+1}+2\al_{n+2} > 2 \al_{n+2} > 2^{-n-1}$, see Figure~\ref{fig2}.
The square $Q^*$ contains four squares of sidelengths $\al_{n+2}$ from the
set $K_{n+2}$, each of them belonging to a different component of $K_{n+1}$.

Let $u_Q$ be a Lipschitz function supported in $Q$ and such that $u_Q=1$
on $Q^*$ and $|\grad u_Q|\le 2/(\alp_n - \be_n)=1/(\al_{n+1}-\al_{n+2})$, e.g.\ 
\[
u_Q(x)=\frac{\max\{1-\dist(x,Q^*),0\}}{\al_{n+1}-\al_{n+2}}.
\]

Then
\[
 \int_\Om |\grad u_Q|^p\,d\La_2 \le \frac{\La_2(Q\cap\Om)}{(\al_{n+1}-\al_{n+2})^p}.
\]
By translation and \eqref{eq-length-Kt}  we see that 
\begin{equation}
   \La_2(Q\cap\Om) = \La_2([0,\al_{n}]^2\setm K)  = \al_n^2 - 4^{-n}  \label{eq-10.2} 
        = 4^{-n} (2c_n+c_n^2) \le 3\cdot 4^{-n} c_n. 
\end{equation}
Since $\al_{n+1}-\al_{n+2} = 2^{-n-2} (1+2c_{n+1}-c_{n+2}) > 2^{-n-2}$, this yields
\begin{equation*}  
\int_\Om |\grad u_Q|^p\,d\La_2 
   \le 3\cdot 2^{p(n+2)}4^{-n} c_{n}.
\end{equation*}
We also have, using  \eqref{eq-10.2} again, that
\[   
\int_\Om |u_Q|^p\,d\La_2 \le \La_2(Q\cap\Om)  \le 3\cdot 4^{-n} c_n.
\]
Next, let
\begin{align*}
K_n^* &= \bigcup_{j=1}^{4^{n}} Q^*_{n,j} \quad \text{and} \quad
u_n =\sum_{j=1}^{4^{n}} u_{Q_{n,j}}, \quad n=0,1,\ldots.
\end{align*}
From the last two estimates we conclude that
\[ 
  \|u_n\|^p_{\Np(\Om)}\le 4^{n} ( 3\cdot 2^{p(n+2)} 4^{-n} c_{n}  
                  + 3\cdot 4^{-n} c_{n})  \le 4\cdot 2^{p(n+2)} c_{n}
\]
and hence
\begin{equation}   \label{eq-norm-un}
    \|u_n\|_{\Np(\Om)} \le 4^{1/p} 2^{n+2}  c_{n}^{1/p} \le 16 \cdot 2^{n}  c_{n}^{1/p}.
\end{equation}
Finally, set
\[ 
K^* = \bigcap_{k=0}^\infty \bigcup_{n=k}^\infty K^*_n
     \quad \text{and} \quad v_k=\sum_{n=k}^\infty u_n, \quad k=0,1,\ldots.
\]
As $K_n^* \subset K_n$ for each $n$, we see that $K^* \subset K$.
Moreover, for each $k$, the function $v_k$ is admissible in the definition
of $\bCp(K^*;\Om)$ and hence by \eqref{eq-norm-un},
\begin{align*}
\bCp(K^*;\Om)^{1/p} &\le \|v_k\|_{\Np(\Om)}
  \le \sum_{n=k}^\infty \|u_n\|_{\Np(\Om)} \le 16 \sum_{n=k}^\infty  2^{n}  c_{n}^{1/p}.
\end{align*}
Choosing $c_n=2^{-n^2}$ we get
\[
\bCp(K^*;\Om)^{1/p} \le 16 \sum_{n=k}^\infty  2^{n-n^2/p}.
\]
The sum in the right-hand side converges,
since $2^{n-n^2/p} < 2^{-n}$ if $n > 2p$, and thus
the right-hand side tends to $0$ as $k \to \infty$.
Hence $\bCp(K^*;\Om)=0$.
(In fact this is true for all $p>1$ as long as $(\log c_n)/n \to -\infty$
as $n \to \infty$.) 

It remains to show that $K^*$ has area $1$.
Let $k\ge0$ be fixed and consider the set
\begin{equation}   \label{eq-K-cap-K*}
     K \cap \bigcup_{n=k}^\infty K^*_n. 
\end{equation}
The $(k+2)$-th generation of $K$ is made up of $4^{k+2}$ parts and $K \cap  K^*_k$ 
consists of exactly $4^{k+1}$ of these parts.  Hence 
\[
     \La_2(K \cap K^*_k) = \frac{4^{k+1}}{4^{k+2}} \La_2(K) = \frac14. 
\]
Similarly, the $(k+3)$-th generation of $K$ is made up of $4^{k+3}$ parts and 
$K \cap  K^*_{k+1}$ consists of exactly $4^{k+2}$ of these parts, however one fourth
of those are already contained in $K \cap  K^*_k$, and thus 
\[
     \La_2(K \cap (K^*_{k+1} \setm K^*_k)) = \frac{3\cdot 4^{k+1}}{4^{k+3}} \La_2(K) 
     = \frac{3}{4}\cdot \frac14. 
\]
Proceeding in the same way, we see that
\[
     \La_2\biggl(K \cap \biggl(K^*_{m} \setm \bigcup_{n=k}^{m-1} K^*_n\biggr)\biggr) 
     = \frac{3^{m-k}\cdot 4^{k+1}}{4^{m+2}} \La_2(K) 
     = \biggl(\frac{3}{4}\biggr)^{m-k}  \frac14, \quad m=k,k+1,\ldots. 
\]
Summing up we obtain 
\[
     \La_2\biggl(K \cap \bigcup_{n=k}^{\infty} K^*_n\biggr)
       = \biggl(1 + \frac{3}{4} + \biggl(\frac{3}{4}\biggr)^2 + \ldots\biggr) \frac{1}{4} = 1.
\]
From which we conclude that
\[
    \La_2(K^*)=\La_2(K \cap K^*) = \lim_{k \to \infty} 
     \La_2\biggl(K \cap \bigcup_{m=k}^{\infty} K^*_m\biggr) = 1.
\]

On the other hand, for $p>1$, we have $\bCp(K;\Om)\ge \bCp(\Kt\times\{0\};\Om) >0$,
since every $u\in\Np(\Om)$ admissible in the definition of 
$\bCp(\Kt\times\{0\};\Om)$ gives by reflection in the $x$-axis rise to 
$v\in\Np((-1,3) \times (-1,1))$ with
\[
    \|v\|^p_{\Np((-1,3) \times (-1,1))} \le 2 \|u\|^p_{\Np(\Om)}.
\]
It is well known (see e.g.\ Theorem~2.26 in Heinonen--Kilpel\"ainen--Martio~\cite{HeKiMa})
that sets of \p-capacity zero in $\R^n$ have Hausdorff dimension at most
$n-p$ if $p \le n$. (If $n >p$ we instead use that the \p-capacity is zero only for 
the empty set.) Hence as $\La_1(\Kt)>0$, we conclude that $\bCp(\Kt\times\{0\};\Om) >0$.

On the contrary, the set $K^*$ constructed above has positive area (and full measure in $K$)
but zero $\bCp(\,\cdot\,;\Om)$-capacity for all $p\ge1$.
\end{example}

In view of the above examples, the following is a natural question to ask.

\begin{openprob}
Assume that $\Om$ is as in Section~\ref{sect-bdyM-Om} and let $E \subset \bdy \Om$ be 
the set of inaccessible boundary points, see below. Is it then true that
\begin{enumerate}
\item \label{op-i}
  $\bCp(E;\Om)=0$, and that
\item \label{op-ii}
  $\Hp f = \Hp (f+h)$ for all $f \in C(\bdy \Om)$ and $h : \bdy \Om \to \eR$ 
such that $h=0$ on $\bdy \Om \setm E$?
\end{enumerate}
\end{openprob}

A point $x \in \bdy \Om$ is \emph{inaccessible} if there is
no curve $\ga:[0,1] \to \overline{\Om}$ such that $\ga([0,1)) \subset \Om$ and 
$\ga(1)=x$. Here $\ga$ is \emph{not} required to be rectifiable.

A positive answer to \ref{op-i} directly yields a positive
answer to \ref{op-ii}, by Proposition~\ref{prop-cts-resolve-Om}.

In the linear case $p=2$ on unweighted $\R^n$, part~\ref{op-ii} is true. 
This can be seen by observing that the Perron solution at a point $y \in \Om$ is the
expected value of the first point $x\in\bdry\Om$ which the Brownian
motion (starting at $y$) hits, and this point is almost surely not in $E$.
We are not aware of any nonprobabilistic proof of this fact.

In Example~\ref{ex-comb-wide-thin} we saw that both~\ref{op-i}
and~\ref{op-ii} are true for the topologist's comb in the nonlinear case.
For the topologist's comb a more general invariance result is obtained in
A.~Bj\"orn~\cite{ABcomb}.

\section{Generalized Perron solutions for domains in \texorpdfstring{$\Omm$}{}}
\label{sect-gen-Perron}

\emph{We assume, in this section, that $G$ is a 
bounded domain which is finitely connected at the boundary 
and that\/ $\Om \subset G$ is a nonempty open subset
with $\Cp(X \setm \Om)>0$. Recall also the standing assumptions from the 
end of Section~\ref{sect-newcap}.}
\medskip

Let us take another look at Examples~\ref{ex-comb-wide-thin} and~\ref{ex-double-cone}.
In Example~\ref{ex-double-cone} there are two directions of reaching every boundary point
in the slits 
\[
   A=\{0\}\times(0,1] \quad \text{and} \quad S^\pm_j=\{\pm2^{-j}\}\times (0,1], \quad j=1,2,\ldots,
\]
but since $\Om$ is not finitely connected at the boundary we cannot use
the Mazurkiewicz boundary results from Sections~\ref{sect-Perron-Omm}
and~\ref{sect-bdyM-Om}.  A way around it is to consider $\Om$ as a subdomain of a
larger open set $G$, and equip $\Om$ with the restriction of its 
Mazurkiewicz distance. In this section we consider such an approach.
Since we  now have to deal with two open sets, the notation becomes
more cumbersome,  which is why we avoided this generality in 
Sections~\ref{sect-Perron-Omm} and~\ref{sect-bdyM-Om}.
However, there are \emph{no} additional technical difficulties.

We equip $G$ with its Mazurkiewicz distance and consider its
closure $\clGM$. This closure is compact by Theorem~\ref{thm-clOmm-cpt},
since $G$ is finitely connected at the boundary. Throughout this section the Mazurkiewicz 
distance is \emph{always} taken with respect to $G$,
and to avoid misunderstandings we write $\dGM$ instead of  $\dM$. We  also 
write $\OmGm$, $\clOmGm$ and $\bdyGM\Om$ when we equip $\Om$ with the 
distance $\dGM$ and take its closure and boundary in $\clGM$.

\begin{deff}  \label{def-Perron-VOmm}
Given a function $f : \bdyGM\Om \to \eR$, let $\UU_f(\OmGm)$ be the set of all 
superharmonic functions $u$ on $\Om$, bounded from below, such that 
\[ 
   \liminf_{\Om \ni y \dGMto x} u(y) \ge f(x) \quad \text{for all } x \in \bdyGM\Om.
\] 
The \emph{generalized upper Perron solution} of $f$ is defined by
\[ 
    \uHpind{\OmGm} f (x) = \inf_{u \in \UU_f(\OmGm)}  u(x), \quad x \in \Om.
\]
The \emph{generalized lower Perron solution} $\lHpind{\OmGm} f$ is
defined similarly, or by $\lHpind{\OmGm} f = - \uHpind{\OmGm}(-f)$.

If $\uHpind{\OmGm} f = \lHpind{\OmGm} f$, then we let  $\Hpind{\OmGm} f := \uHpind{\OmGm} f$
and $f$ is said to be \emph{resolutive} with respect to $\OmGm$.
\end{deff}

The results in Sections~\ref{sect-Perron-Omm} and~\ref{sect-bdyM-Om}
can all be formulated in this generality, and the proofs remain the same. Let us 
formulate these results.  

\begin{thm} \label{thm-Np-invariance-VOmm}
Assume that $h: \bdyGM \Om \to \eR$ is zero $\bCp(\,\cdot\,;\OmGm)$-q.e., 
If either $f \in C(\bdyGM\Om)$, or $f: \clOmGm \to \eR$ is a bounded 
$\bCp(\,\cdot\,;\OmGm)$-quasi\-con\-tin\-u\-ous function such that $f|_{\Om} \in \Np(\Om)$, 
then $f$ and  $f+h$ are resolutive with respect to\/ $\OmGm$ and
\[
     \Hpind{\OmGm} (f+h) = \Hpind{\OmGm} f = \oHpind{\Om} f.
\]
\end{thm}

\begin{cor} \label{cor-unique-VOmm}
Assume that either $f \in C(\bdyGM\Om)$, or that $f: \clOmGm \to \eR$ is a bounded 
$\bCp(\,\cdot\,;\OmGm)$-quasi\-con\-tin\-u\-ous function
such that $f|_{\Om} \in \Np(\Om)$. If $u$ is a bounded \p-harmonic function
in\/ $\Om$ and if there is a set  $E \subset \bdyGM\Om$ with $\bCp(E;\OmGm)=0$ such that
\[
    \lim_{\Om \ni y \dGMto  x} u(y) = f(x) \quad \text{for all } x \in \bdyGM\Om \setm E,
\]
then $u=\Hpind{\OmGm} f$.
\end{cor}

As already mentioned, the proofs of these results are the same as the 
proofs given in Section~\ref{sect-Perron-Omm}. 
Let us just point out that the following fundamental
equality follows from Proposition~\ref{prop-Np0},
\begin{align*}
    \Np_0(\Om)& =\{f|_\Om : f  \in \Np_0(G) 
                    \text{ and $f=0$ in } G \setm \Om\} \\
          & =\{f|_\Om : f  \in \Np_0(G^M) 
                    \text{ and $f=0$ in } G \setm \Om\} 
           =\Np_0(\Om;\clGM).
\end{align*}
Note also that the above results could not be obtained as direct
consequences of the results in Section~\ref{sect-bdyM-Om} by replacing 
$X$ with $\clGM$ since the space $\clGM$ need not satisfy the standing
assumptions about doubling and Poincar\'e inequality.

We end this section with a demonstration of the described
technique in the case of the double comb. It makes it possible to treat discontinuities 
at finitely many slits $S_j$,  including the central one $A$.
This procedure can be iterated by adding more and more open slits, 
so that finally the whole comb can be treated, see A.~Bj\"orn~\cite{ABcomb}
for more details. 

\begin{example} \label{ex-gen-Perron-double-comb}
Let $\Om \subset \R^2$ be the double comb as in Example~\ref{ex-double-cone} and let 
\[
G=(-1,1)^2 \setm (\itoverline{A}\cup\bigcup_{j\in J}\itoverline{S}^\pm_j),
\] 
where $J$ is a finite set of indices and $A=\{0\}\times(0,1]$.
Equip $G$ with its Mazurkiewicz distance $\dGM$
and  consider $\Om$ as an open subset of $G^M$ in this new metric.
Observe that $G$ is finitely connected at  the boundary.
The  boundary $\bdyGM\Om$ will be the same as $\bdym\Om$ apart from
that the boundary points in $A\cup\bigcup_{j\in J}S^\pm_j$
will be split into two boundary points each.

Let $f$ be the ``jump'' function from \eqref{eq-ex-double-cone},
and set $\ft=0$ on the left copy of $A$ and $\ft=f$ otherwise.
Then $\ft\in\Np(\clOmGm)$ is continuous in $\clOmGm$. 
Thus, it is resolutive with respect to $\OmGm$ by 
Theorem~\ref{thm-Np-invariance-VOmm},  and can be perturbed arbitrarily 
on (both copies) of $A$ without changing the obtained Perron solution.
Since $\bCp(A,\OmGm)=0$ and $f=\ft$ except at the left copy of $A$, 
the resolutivity of $f$ with respect to $\OmGm$ follows.
Finally, it is easy to see that $\Hpind{\Om} f = \Hpind{\OmGm} f$, 
and thus $f$ is resolutive also with respect to $\Om$.

Contrary to Example~\ref{ex-double-cone}, this method also allows
us to treat functions with different values on the left and right copies
of the slits $S^\pm_j$, $j\in J$, e.g.\ similar to the ``thickened'' slits 
in Example~\ref{ex-comb-wide-thin}.
\end{example}

\section*{Appendix. Comparison of capacities}

\setcounter{equation}{0}
\setcounter{section}{1}
\setcounter{thm}{0}
\renewcommand{\thesection}{\textup{\Alph{section}}}%
\addcontentsline{toc}{section}{Appendix}

\emph{In this appendix we do {\bf not} require the assumptions
from the end of Section~\ref{sect-newcap} to hold.}

\medskip

The focus of this appendix is to compare the new capacity with the two natural capacities on 
$\overline{\Om}$ equipped with $\mu$ resp.\  $\mu_0$.

\begin{prop} 
If $E \subset \overline{\Om}$, then
\[
     \Cp(E;(\overline{\Om};\mu_0)) \le  \Cp(E;\overline{\Om}) \le \Cp(E).
\]
\end{prop}

This is trivial, and it is easy to find examples where strict
inequalities hold by e.g.\  Examples~\ref{ex-Cantor-2} and~\ref{ex-cusp-new} respectively.

For the new capacity the situation is a little more complicated.

\begin{prop} \label{prop-newcap-compare-1}
Let $G \subset \overline{\Om}$ be relatively open. Then
\[
	\bCp(G;\Om) \le      \Cp(G;(\overline{\Om};\mu_0)).
\]
If $X$ is proper and continuous functions are dense in $\Np(\clOm,\mu_0)$, then 
\[
   \bCp(E;\Om) \le      \Cp(E;(\overline{\Om};\mu_0)) \quad \text{for all } E\subset\clOm.
\]
\end{prop}

\begin{proof}
Since $G$ is relatively open in $\overline{\Om}$, every function admissible in the definition
of $\Cp(G;(\clOm,\mu_0))$ is also admissible for $\bCp(G;\Om)$.
Taking infimum over all such functions proves the first part.

As for the last part,  we may assume that $\Cp(E;(\overline{\Om};\mu_0))<\infty$.
Let $\eps>0$.  By Theorem~\ref{thm-quasicont} there is an open $G \supset E$ such that
$\Cp(G;(\overline{\Om};\mu_0)) < \Cp(E;(\overline{\Om};\mu_0)) + \eps$.
By monotonicty and the first part we see that
\[
  \bCp(E;\Om) \le \bCp(G ;\Om)\le \Cp(G;(\overline{\Om};\mu_0)) < \Cp(E;(\overline{\Om};\mu_0)) + \eps.
\]
Letting $\eps \to 0$ concludes the proof.
\end{proof}

Example~\ref{ex-comb-wide-thin} (together with 
Proposition~\ref{prop-outercap}) shows that the inequality 
in Proposition~\ref{prop-newcap-compare-1} can be strict. The following result shows 
that at the level of null sets Proposition~\ref{prop-newcap-compare-1} holds for all sets without 
any continuity assumptions.

\begin{prop} \label{prop-newcap-compare-0}
Assume that $X$ is proper. If $E \subset \overline{\Om}$ and
$\Cp(E;(\overline{\Om};\mu_0))=0$, then $\bCp(E;\Om)=0$.
\end{prop}

Examples~\ref{ex-comb-wide-thin}, \ref{ex-double-cone} and~\ref{ex-Cantor} 
all show that the converse implication does not hold. Recall that 
if $\mu$ is doubling then $X$ is complete if and only if $X$ is proper.

\begin{proof}
Let $\eps >0$. As $\overline{\Om}$ is proper, Proposition~\ref{prop-zero-cap-capacitable} 
shows that there is a relatively open set $G \supset E$ with
$\Cp(G;(\overline{\Om};\mu_0))< \eps$. By Theorem~\ref{prop-newcap-compare-1},
\[
   \bCp(E;\Om) \le \bCp(G;\Om) < \eps.
\]
Letting $\eps \to 0$ completes the proof.
\end{proof}

The next proposition follows directly from Proposition~\ref{prop-newcap-compare-1},
since in the definition of quasicontinuity we only consider
relatively open subsets of $\overline{\Om}$.

\begin{cor} \label{prop-qcont-imp}
Assume that $f \in \overline{\Om} \to \eR$ is
quasicontinuous with respect to\/ $(\overline{\Om};\mu_0)$.
Then $f$ is $\bCp(\,\cdot\,;\Om)$-quasi\-con\-tin\-u\-ous.
\end{cor}

We can also improve upon Lemma~\ref{lem-Phicap-Omm}.

\begin{prop} \label{prop-Phicap-Omm-equal}
Assume that $X$ is proper and locally connected,
and that\/ $\Om$ is a bounded domain which is finitely connected at the boundary.
Let $E \subset \overline{\Om}$. Then 
\[ 
	\bCp(\Phi^{-1}(E);\Omm)
	= \bCp(E;\Om).
\]
\end{prop}

\begin{proof}
The inequality $\bCp(\Phi^{-1}(E);\Omm) \le \bCp(E;\Om)$
was proved in Lemma~\ref{lem-Phicap-Omm}.
(Note that the proof of the first inequality in Lemma~\ref{lem-Phicap-Omm} 
only requires $X$ to be locally connected, through the equality
$\Np(\Om)=\Np(\Omm)$.) For the converse inequality,  assume
that $\bCp(\Phi^{-1}(E);\Omm)<\infty$  and let $u \in \A_{\Phi^{-1}(E)}$.
For $x \in E \cap \bdy \Om$ let  $x_j \in \Om$, $j=1,2,\cdots$, be a sequence of points such 
that $x_j \to x$ in the metric $d$ as $j \to \infty$.

We shall show that $\liminf_{j \to \infty} u(x_j) \ge 1$. Assume not.
Then there is a subsequence, also denoted $\{x_j\}_{j=1}^\infty$,
such that $\lim_{j \to \infty} u(x_j)<1$. 
By the compactness of $\clOmm$ (see Theorem~\ref{thm-clOmm-cpt} and the
comment after it), the sequence $\{x_j\}_{j=1}^\infty$ 
has a convergent subsequence $\{x_{j_k}\}_{k=1}^\infty$ tending 
in the metric $\dM$ to some point $x_0 \in \clOmm$. As $\Phi$ is Lipschitz on $\clOmm$, we have
\[
    \Phi(x_0) = \lim_{k \to \infty} \Phi(x_{j_k})  =  \lim_{k \to \infty} x_{j_k} = x,
\]
i.e.\ $x_0\in\Phi^{-1}(E)$. Since $u \in \A_{\Phi^{-1}(E)}$, it follows that
\[
      \liminf_{k \to \infty} u(x_{j_k}) \ge 1,
\]
contradicting $\lim_{j \to \infty} u(x_j)<1$. Thus $\liminf_{j \to \infty} u(x_j) \ge 1$ and
$u \in \A_E$ (observe that $\Np(\Om)=\Np(\Omm)$ by the discussion
at the beginning of Section~\ref{sect-NpOmm}.). Hence 
\[
     \bCp(E;\Om)  \le \|u\|_{\Np(\Om)}^p
\]
and taking infimum over all $u\in\A_{\Phi^{-1}(E)}$ completes the proof.
\end{proof}

\end{document}